\documentclass[12pt]{article}
\usepackage{amssymb,amsmath,amsthm, amsfonts}
\usepackage{graphicx}
\usepackage{epsfig}
\usepackage{tikz}
\usepackage{mathrsfs}

\def\disp{\displaystyle}

\def\dref#1{(\ref{#1})}

\theoremstyle{plain}
\newtheorem{theorem}{Theorem}[section]
\newtheorem{lemma}{Lemma}[section]

\theoremstyle{definition}
\newtheorem{definition}{Definition}[section]

\newtheorem{remark}{Remark}[section]

\numberwithin{equation}{section}

\begin{document}

\title{\bf Global weak solutions in a three-dimensional Keller-Segel-Navier-Stokes system modeling coral fertilization
}

\author{
Jiashan Zheng\thanks{Corresponding author.   E-mail address:
 zhengjiashan2008@163.com (J.Zheng)}
 \\
    School of Mathematics and Statistics Science,\\
     Ludong University, Yantai 264025,  P.R.China \\
}
\date{}


\maketitle \vspace{0.3cm}
\noindent
\begin{abstract}
We consider an initial-boundary value problem for the incompressible four-component Keller-Segel-Navier-Stokes system  with  rotational flux
$$
 \left\{
 \begin{array}{l}
   n_t+u\cdot\nabla n=\Delta n-\nabla\cdot(nS(x,n,c)\nabla c)-nm,\quad
x\in \Omega, t>0,\\
    c_t+u\cdot\nabla c=\Delta c-c+m,\quad
x\in \Omega, t>0,\\
 m_t+u\cdot\nabla m=\Delta m-nm,\quad
x\in \Omega, t>0,\\
u_t+\kappa(u \cdot \nabla)u+\nabla P=\Delta u+(n+m)\nabla \phi,\quad
x\in \Omega, t>0,\\
\nabla\cdot u=0,\quad
x\in \Omega, t>0\\
 \end{array}\right.\eqno(1.1)
 $$
 in a bounded domain $\Omega\subset \mathbb{R}^3$ with smooth boundary, where $\kappa\in \mathbb{R}$ is given
constant, $S$ is a matrix-valued sensitivity satisfying $|S(x,n,c)|\leq C_S(1+n)^{-\alpha}$ with some
$C_S> 0$ and $\alpha\geq 0$.   
As the case $\kappa = 0$ (with $\alpha\geq\frac{1}{3}$ or the initial data satisfy a certain smallness condition)
has been considered in \cite{Lidfff00}, based on new gradient-like functional
 inequality, it is shown in the present
paper that the corresponding initial-boundary problem with $\kappa \neq 0$ admits at least one global weak solution
if $\alpha>0$.  To the best of our knowledge, this is the first analytical work for the {\bf full  three-dimensional four-component} chemotaxis-Navier-Stokes system.

\end{abstract}

\vspace{0.3cm}


\newpage
\section{Introduction}

Many phenomena, which appear in natural science, especially, biology and
physics, support animals' lives (see \cite{Liggh00,Xu5566r793,Wangsseeess21215,Guggg1215}). Chemotaxis has been
extensively studied in the context of modeling mold and bacterial colonies (see Hillen and Painter \cite{Hillen} and Bellomo et al. \cite{Bellomo1216}).
 In order to describe this biological phenomenon in mathematics, in 1970, Keller and Segel (\cite{Keller2710})
proposed the following system
\begin{equation}
 \left\{\begin{array}{ll}
 n_t=\Delta n-\chi\nabla\cdot( n\nabla c),\\
 \disp{ c_t=\Delta c- c +n,}
 \end{array}\right.\label{722dff344101.2xddff16677}
\end{equation}
which is called Keller-Segel system. Here $\chi > 0$ is called chemotactic sensitivity, $n$ and $c$ denote the density of the cell population and the
concentration of the attracting chemical substance, respectively.
 Since then, there has been an enormous amount of effort devoted to the possible blow up and
regularity of solutions, as well as the asymptotic behavior and other properties (see e.g. \cite{Bellomo1216}). We refer to \cite{Hillen,Horstmann2710} and \cite{Perthame793} for the further
reading. Beyond this,
 a large number of variants of system \dref{722dff344101.2xddff16677} have been investigated, including the system with the logistic terms (see \cite{Cao,Tello710,Winkler21215,Zhengssdddssddddkkllssssssssdefr23},
for instance) and  the nonlinear diffusion (\cite{Tao794,Winkler72,Zhengssdefr23,Zheng00,Zheng33312186,Zhengsddfffsdddssddddkkllssssssssdefr23}), the signal is  consumed by the cells (see e.g.
Tao and  Winkler \cite{Tao71215},    \cite{Zhengssssssdefr23})
 two-species chemotaxis system (see \cite{LiLiLingssssssdefr23,Zhenhhhhgssssssssdefr23}, for instance)  and so on.

In order to
discuss of the coral fertilization,
Kiselev and Ryzhik (\cite{Kiselevdd793} and \cite{Kiselevsssdd793}) investigated the important effect of
chemotaxis on the coral fertilization process via the Keller-Segel type system of the form
\begin{equation}
 \left\{\begin{array}{ll}
 \rho_t+u\cdot\nabla\rho=\Delta \rho-\chi\nabla\cdot( \rho\nabla c)-\rho^q,
 \quad
\\
 \disp{ 0=\Delta c +\rho,}\quad\\
 \end{array}\right.\label{722dff344101.dddddgghggghhff2ffggffggx16677}
\end{equation}
where $\rho$ is the density of egg (sperm) gametes, $u$ is the smooth divergence free sea fluid
velocity and
 $c$ denotes the concentration of  chemical signal which is released  by the eggs.  This model \dref{722dff344101.dddddgghggghhff2ffggffggx16677} implicitly assumes that the
densities of sperm and egg gametes are identical.
Kiselev and Ryzhik (\cite{Kiselevdd793} and \cite{Kiselevsssdd793})
proved  that if $q > 2$ and  the chemotactic sensitivity $\chi$ increases, for the associated Cauchy problem of  \dref{722dff344101.dddddgghggghhff2ffggffggx16677},
the total mass
$\int_{\mathbb{R}^2} \rho$ can become arbitrarily
small,  whereas if $q = 2$, a corresponding weaker
but yet relevant effect within finite time intervals is detected (see Kiselev and Ryzhik \cite{Kiselevsssdd793}).
%
%
%
%
%
%
%
%

In various situations, however, the interaction of chemotactic movement of the gametes and the surrounding fluid is not negligible (see   Espejo and Suzuki
\cite{Espejoss12186}, Espejo and Winkler  \cite{EspejojjEspejojainidd793}). To model such biological processes, Espejo and Suzuki
(\cite{Espejoss12186}) proposed the following model
\begin{equation}
 \left\{\begin{array}{ll}
 \rho_t+u\cdot\nabla\rho=\Delta \rho-\chi\nabla\cdot( \rho\nabla c)-\mu\rho^2,
 \quad
\\
c_t+u\cdot\nabla c=\Delta c-c+\rho,
 \quad
\\
u_t+\kappa (u\cdot\nabla)u=\Delta u-\nabla P+\rho\nabla\rho,
 \quad
\\
 \disp{ \nabla\cdot u=0,}\quad\\
 \end{array}\right.\label{722dff344ddd101.ddgghggghhff2ffggffggx16677}
\end{equation}
where 
$\rho$ and $c$ are defined as before. 
 Here $u,P,\phi$ and $\kappa\in \mathbb{R}$ denote, respectively, the velocity field, the associated pressure of the fluid, the potential of the
 gravitational field and the
strength of nonlinear fluid convection.


Recently,
in order  to analyze a further refinement of the model \dref{722dff344ddd101.ddgghggghhff2ffggffggx16677} which
explicitly distinguishes between sperms and eggs,   Espejo and Winkler (\cite{EspejojjEspejojainidd793})
 proposed the following  four-component Keller-Segel(-Navier)-Stokes system  with  (rotational flux):
\begin{equation}
 \left\{\begin{array}{ll}
   n_t+u\cdot\nabla n=\Delta n-\nabla\cdot(nS(x, n, c)\cdot\nabla c)-nm,\quad
x\in \Omega, t>0,\\
    c_t+u\cdot\nabla c=\Delta c-c+m,\quad
x\in \Omega, t>0,\\
 m_t+u\cdot\nabla m=\Delta m-nm,\quad
x\in \Omega, t>0,\\
u_t+\kappa(u \cdot \nabla)u+\nabla P=\Delta u+(n+m)\nabla \phi,\quad
x\in \Omega, t>0,\\
\nabla\cdot u=0,\quad
x\in \Omega, t>0,\\
\disp{(\nabla n-nS(x, n, c))\cdot\nu=\nabla c\cdot\nu=\nabla m\cdot\nu=0,u=0,}\quad
x\in \partial\Omega, t>0,\\
\disp{n(x,0)=n_0(x),c(x,0)=c_0(x),m(x,0)=m_0(x),u(x,0)=u_0(x),}\quad
x\in \Omega\\
 \end{array}\right.\label{1.1}
\end{equation}
in a domain $\Omega\subset  \mathbb{R}^N(N=2)$, where $u,P,\phi$, $\kappa\in \mathbb{R}$ and $c$ are defined as before and
$S$ is a tensor-valued function or a scalar function which satisfies
\begin{equation}\label{x1.73142vghf48rtgyhu}
S\in C^2(\bar{\Omega}\times[0,\infty)^2;\mathbb{R}^{3\times3})
 \end{equation}
 and
 \begin{equation}\label{x1.73142vghf48gg}|S(x, n, c)|\leq C_S(1 + n)^{-\alpha} ~~~~\mbox{for all}~~ (x, n, c)\in\Omega\times [0,\infty)^2
 \end{equation}
with some $C_S > 0$ and $\alpha> 0$.
 Here  the scalar functions $n = n(x, t)$ and $m = m(x, t)$  denote the population densities of unfertilized sperms and
eggs, respectively. In \cite{EspejojjEspejojainidd793}, assuming that $S(x, n, c)\equiv1$, Espejo and Winkler  showed that the 2D  four-component Keller-Segel-Navier-Stokes system \dref{1.1}
 possesses at least one bounded classical  solution,
whereas,
in three dimensions, Li, Pang and  Wang (\cite{Lidfff00}) showed that the four-component Keller-Segel-Stokes ($\kappa=0$ in the first equation of \dref{1.1}) system \dref{1.1}
with  tensor-valued function (where the tensor-valued function   $S$  satisfies \dref{x1.73142vghf48gg} with $\alpha\geq\frac{1}{3}$) possesses at least one bounded classical solution. Recently, by using a (new) weighted estimate,
Zheng (\cite{Zhengssddfffssdddssdddddkkkkfssdddd00}) proved that if $S$  satisfies \dref{x1.73142vghf48gg} with $\alpha>0,$ the four-component Keller-Segel-Stokes system \dref{1.1} admits at least one bounded classical solution.
These indeed extend and improve the recent corresponding results obtained by Li, Pang and  Wang (\cite{Lidfff00}).
%
However,
as far as we know, for the full three-dimensional four-component chemotaxis-Navier-Stokes system \dref{1.1} ($\kappa\not=0$ in \dref{1.1})
it is still not clearly 
whether the solution of the system \dref{1.1} is exists or not.
Moreover, in \cite{EspejojjEspejojainidd793}, \cite{Lidfff00} and \cite{Zhengssddfffssdddssdddddkkkkfssdddd00}, the authors also showed that the corresponding
solutions converge to a spatially homogeneous equilibrium exponentially as $t \rightarrow\infty$ as well.

Motivated by the above works, the main objective of the paper is to  investigate the four-component Keller-Segel-Navier-Stokes system  \dref{1.1}
with  rotational flux.
 We sketch here the main ideas and methods used in this article.
A key role in our existence analysis is played by the observation that for
appropriate positive constants $a_i$ and $b_i(i=1,2)$, the functional
$$
\left\{
\begin{array}{rl}
\disp{\int_{\Omega} n_{\varepsilon}^{4\alpha+\frac{2}{3}}+a_1\int_{\Omega}   |\nabla c_{\varepsilon}|^2+b_1\int_{\Omega}  | {u_{\varepsilon}}|^2~~~\mbox{if}~~\alpha\neq\frac{1}{12},}\\
\disp{\int_{\Omega} n_{\varepsilon}\ln n_{\varepsilon}+a_2\int_{\Omega}   |\nabla c_{\varepsilon}|^2+b_2\int_{\Omega}  | {u_{\varepsilon}}|^2~~~~\mbox{if}~~\alpha=\frac{1}{12}}
\end{array}
\right.
$$
possesses a favorable entropy-like property, where  $n_{\varepsilon},c_{\varepsilon}$ and $u_{\varepsilon}$ are components of the solutions to  \dref{1.1}.
This will entail a series of a
priori estimates which will  derive further $\varepsilon$-independent bounds for spatio-temporal integrals of the approximated solutions and several $\varepsilon$-independent regularity features of their time derivatives (see Section 4-5). On the basis of the compactness properties thereby implied,  we shall finally
pass to the limit along an adequate sequence of numbers $\varepsilon = \varepsilon_j\searrow0$ and thereby verify the main results (see Section 6).

Before going into our mathematical analysis, we recall some important progresses on system \dref{1.1} and
its variants.
In order to  describe the behavior of bacteria of the species Bacillus
subtilis suspended in sessile water drops, Tuval et al. (\cite{Tuval1215}) proposed
  the
following chemotaxis--fluid model
\begin{equation}
 \left\{\begin{array}{ll}
   n_t+u\cdot\nabla n=\Delta n-\nabla\cdot( nS(x,n,c)\nabla c),\quad
x\in \Omega, t>0,\\
    c_t+u\cdot\nabla c=\Delta c-nf(c),\quad
x\in \Omega, t>0,\\
u_t+\kappa (u\cdot\nabla)u+\nabla P=\Delta u+n\nabla \phi,\quad
x\in \Omega, t>0,\\
\nabla\cdot u=0,\quad
x\in \Omega, t>0,\\
 \end{array}\right.\label{1.1hhjffggjddssggtyy}
\end{equation}
where $f(c)$ is the {\bf consumption} rate of the oxygen by the cells.
%
%
%
 The model \dref{1.1hhjffggjddssggtyy} occurs in the modelling of the collective behaviour of chemotaxis-driven swimming aerobic bacteria.
%
%
%
%

If the chemotactic sensitivity $S(x, n, c):=S(c)$ is  a scalar function, by making use of energy-type functionals, some
local and global solvability of corresponding initial value problem for \dref{1.1hhjffggjddssggtyy} in either bounded or unbounded domains have
been obtained in the past years (see e.g.
Chae et. al. \cite{Chaexdd12176},
Duan et. al. \cite{Duan12186},
Liu and Lorz  \cite{Liu1215,Lorz1215},
 Tao and Winkler   \cite{Tao41215,Winkler31215,Winkler61215,Winkler51215}, Zhang and Zheng \cite{Zhang12176},  Zheng \cite{Zhengssddfffssdddssdddddkkkkfssdddd00} and references therein).

 As pointed out by Xue and Othmer  in  \cite{Xusddeddff345511215}, the chemotactic sensitivity $S$ should be a tensor function rather than a scalar one, so that,
 the corresponding chemotaxis-fluid system \dref{1.1hhjffggjddssggtyy} loses some energy-like structure, which plays a key role in the analysis of the scalar-valued case. Therefore, there are only a few works concerning chemotaxis-fluid coupled models with tensor-valued sensitivity (see  Ishida \cite{Ishida1215},  Wang et al. \cite{He76,Wang11215,Wang21215}, Winkler \cite{Winkler11215} and
 Zheng \cite{Zhengssddfffssdddssdddddkkkkfssdddd00} for  example).

 In comparison to \dref{1.1hhjffggjddssggtyy}, if  we assume that
  the signal is produced other than consumed by cells, then the corresponding
  chemotaxis-fluid model is the  Keller-Segel-fluid system
of the form (see \cite{Winkler444ssdff51215,Wang23421215,Wang21215,Wangss21215,Zhenddddgssddsddfff00,Kegssddsddfff00})
\begin{equation}
 \left\{\begin{array}{ll}
   n_t+u\cdot\nabla n=\Delta n-\nabla\cdot( n S(x,n,c)\cdot\nabla c),\quad
x\in \Omega, t>0,\\
    c_t+u\cdot\nabla c=\Delta c-c+n,\quad
x\in \Omega, t>0,\\
u_t+\kappa (u\cdot\nabla)u+\nabla P=\Delta u+n\nabla \phi,\quad
x\in \Omega, t>0,\\
\nabla\cdot u=0,\quad
x\in \Omega, t>0. \\
 \end{array}\right.\label{1sdfdffgggggsxdcfffggvgb.1}
\end{equation}
Over the past few years, the mathematical analysis of \dref{1sdfdffgggggsxdcfffggvgb.1} (with tensor-valued
sensitivity) began to flourish (see \cite{Winkler444ssdff51215,Wang23421215,Wang21215,Wangss21215,Zhenddddgssddsddfff00,Kegssddsddfff00} and references therein). In fact, if the domain $\Omega$ is
further assumed to be convex,
Wang, Xiang and Winkler (\cite{Wang23421215})
%
 established the global existence and boundedness of
the 2D system \dref{1sdfdffgggggsxdcfffggvgb.1} under the assumption of \dref{x1.73142vghf48gg} with $\alpha >0$.
Recently, Zheng (\cite{Zhengssdddssdddddkkkkfssdddd00})   extends   the results of   \cite{Wang23421215} to the  general bounded domain by  some new
entropy-energy estimates. Furthermore, if $S(x, n, c)$ satisfying \dref{x1.73142vghf48rtgyhu} and \dref{x1.73142vghf48gg} with $\alpha > \frac{1}{2}$, Wang and Xiang (\cite{Wangss21215}) proved the same result for
for the
three-dimensional Stokes version ($\kappa=0$ in the first equation of \dref{1.1}) of system \dref{1.1}.
Wang and Liu (\cite{LiuZhLiuLiuandddgddff4556}) showed
 that 3D Keller-Segel-Navier-Stokes ($\kappa\neq0$ in the first equation of \dref{1.1})  system \dref{1.1} admits a global weak solutions for
 tensor-valued sensitivity
$S(x, n, c)$ satisfying \dref{x1.73142vghf48rtgyhu}
 and \dref{x1.73142vghf48gg} with $\alpha > \frac{3}{7}$. More
recently, Ke and Zheng (\cite{Kegssddsddfff00})
 extends the result of \cite{LiuZhLiuLiuandddgddff4556} to the case
$\alpha > \frac{1}{3}$,
which in light of the known
results for the fluid-free system mentioned above is an optimal restriction on $\alpha$.
Some other  results on global existence and boundedness properties have also been obtained for the variant of
\dref{1sdfdffgggggsxdcfffggvgb.1} obtained on replacing $\Delta n$ by nonlinear diffusion operators generalizing the porous medium-type
choice $\Delta n^m$ for several ranges of $m > 1$ (\cite{Zhenddddgssddsddfff00,Peng55667,Li33223321215,Zhengsddfffsdddssddddkkllssssssssdefr23}).



%
%
%



In order to formulate our main result, we will first briefly introduce the technique framework:
 The initial
data are assumed to be
\begin{equation}\label{ccvvx1.731426677gg}
\left\{
\begin{array}{ll}
\displaystyle{n_0\in C(\bar{\Omega})~~~~ \mbox{with}~~ n_0\geq0 ~~\mbox{and}~~n_0\not\equiv0},\\
\displaystyle{c_0\in W^{1,\infty}(\Omega)~~\mbox{with}~~c_0\geq0~~\mbox{in}~~\bar{\Omega},}\\
\displaystyle{m_0\in C(\bar{\Omega})~~~~ \mbox{with}~~ m_0\geq0 ~~\mbox{and}~~m_0\not\equiv0},\\
\displaystyle{u_0\in D(A^\gamma_{r})~~\mbox{for~~ some}~~\gamma\in ( \frac{3}{4}, 1)~~\mbox{and any}~~ {r}\in (1,\infty),}\\
\end{array}
\right.
\end{equation}
where $A_{r}$ denotes the Stokes operator with domain $D(A_{r}) := W^{2,{r}}(\Omega)\cap  W^{1,{r}}_0(\Omega)
\cap L^{r}_{\sigma}(\Omega)$,
and
$L^{r}_{\sigma}(\Omega) := \{\varphi\in  L^{r}(\Omega)|\nabla\cdot\varphi = 0\}$ for ${r}\in(1,\infty)$
 (\cite{Sohr}).
 Apart from this, we shall merely suppose that
 \begin{equation}
\phi\in W^{2,\infty}(\Omega).
\label{dd1.1fghyuisdakkkllljjjkk}
\end{equation}

Under these assumptions, our main result can be read as

\begin{theorem}\label{theorem3}
Let  $\Omega\subset \mathbb{R}^3$ be a bounded    domain with smooth boundary,
 \dref{dd1.1fghyuisdakkkllljjjkk} and \dref{ccvvx1.731426677gg}
 hold, and suppose that $S$ satisfies \dref{x1.73142vghf48rtgyhu} and \dref{x1.73142vghf48gg}
with some
\begin{equation}\label{x1.73142vghf48}\alpha>0.
\end{equation}
Then 
 the problem \dref{1.1} possesses at least
one global weak solution $(n, c, u, P)$
 in the sense of Definition \ref{df1}. 
\end{theorem}
\begin{remark}
(i)  To the best of our knowledge, this is the first analytical work for the full {\bf three-dimensional four-component} chemotaxis-Navier-Stokes system.


(ii) We should pointed that the idea of this paper can not deal with the case $\alpha=0$, since, it is hard to establish the $\varepsilon$-independent estimates
(see the proof of Lemma \ref{lemmaghjffggssddgghhmk4563025xxhjklojjkkk}).

(iii) We have to leave open the question whether the condition \dref{x1.73142vghf48} is
optimal.

\end{remark}

\section{Preliminaries}
Due to the strongly nonlinear term $\kappa(u \cdot \nabla)u$ and  the presence of tensor-valued $S$ in system \dref{1.1}, we need to consider an
appropriately regularized problem of \dref{1.1} at first.
%
According to the ideas in  \cite{Winkler51215}, the corresponding regularized problem is introduced as follows:
\begin{equation}
 \left\{\begin{array}{ll}
   n_{\varepsilon t}+u_{\varepsilon}\cdot\nabla n_{\varepsilon}=\Delta n_{\varepsilon}-\nabla\cdot(n_{\varepsilon}\frac{1}{(1+\varepsilon n_{\varepsilon})}S_\varepsilon(x, n_{\varepsilon}, c_{\varepsilon})\nabla c_{\varepsilon})-n_{\varepsilon}m_{\varepsilon},\quad
x\in \Omega, t>0,\\
       c_{\varepsilon t}+u_{\varepsilon}\cdot\nabla c_{\varepsilon}=\Delta c_{\varepsilon}-c_{\varepsilon}+m_{\varepsilon},\quad
x\in \Omega, t>0,\\
m_{\varepsilon t}+u_{\varepsilon}\cdot\nabla m_{\varepsilon}=\Delta m_{\varepsilon}-n_{\varepsilon}m_{\varepsilon},\quad
x\in \Omega, t>0,\\
u_{\varepsilon t}+\nabla P_{\varepsilon}=\Delta u_{\varepsilon}-\kappa (Y_{\varepsilon}u_{\varepsilon} \cdot \nabla)u_{\varepsilon}+(n_{\varepsilon}+m_{\varepsilon})\nabla \phi,\quad
x\in \Omega, t>0,\\
\nabla\cdot u_{\varepsilon}=0,\quad
x\in \Omega, t>0,\\
 \disp{\nabla n_{\varepsilon}\cdot\nu=\nabla c_{\varepsilon}\cdot\nu=0,u_{\varepsilon}=0,\quad
x\in \partial\Omega, t>0,}\\
\disp{n_{\varepsilon}(x,0)=n_0(x),c_{\varepsilon}(x,0)=c_0(x),m_{\varepsilon}(x,0)=m_0(x),u_{\varepsilon}(x,0)=u_0(x)},\quad
x\in \Omega,\\
 \end{array}\right.\label{1.1fghyuisda}
\end{equation}
where for $\varepsilon\in (0, 1),$
\begin{equation}
\begin{array}{ll}
S_\varepsilon(x, n, c) = \rho_\varepsilon(x)S(x, n, c),~~ x\in\bar{\Omega},~~n\geq0,~~c\geq0
 \end{array}\label{3.10gghhjuuloollyuigghhhyy}
\end{equation}
and
\begin{equation}
 \begin{array}{ll}
 Y_{\varepsilon}w := (1 + \varepsilon A)^{-1}w ~~~~\mbox{for all}~~ w\in L^2_{\sigma}(\Omega)
 \end{array}\label{aasddffgg1.1fghyuisda}
\end{equation}
is the standard Yosida approximation.
Here $(\rho_\varepsilon)_{\varepsilon\in(0,1)} \in C^\infty_0 (\Omega)$
  be a family of standard cut-off functions satisfying $0\leq\rho_\varepsilon\leq 1$
   in $\Omega$
 and $\rho_\varepsilon\nearrow1$ in $\Omega$
 as $\varepsilon\searrow0$.

By an adaptation of well-established fixed point arguments (see e.g. Lemma 2.1 of \cite{Winkler51215} as well as \cite{Winkler11215} and Lemma 2.1 of \cite{Painter55677})
and a suitable
extensibility criterion, one can readily verify the
local solvability of \dref{1.1fghyuisda}. 

%
%
%
\begin{lemma}\label{lemma70}
Assume
that $\varepsilon\in(0,1).$
%
Then there exist $T_{max,\varepsilon}\in  (0,\infty]$ and
a classical solution $(n_\varepsilon, c_\varepsilon, u_\varepsilon, P_\varepsilon)$ of \dref{1.1fghyuisda} in
$\Omega\times(0, T_{max,\varepsilon})$ such that
\begin{equation}
 \left\{\begin{array}{ll}
 n_\varepsilon\in C^0(\bar{\Omega}\times[0,T_{max,\varepsilon}))\cap C^{2,1}(\bar{\Omega}\times(0,T_{max,\varepsilon})),\\
  c_\varepsilon\in  C^0(\bar{\Omega}\times[0,T_{max,\varepsilon}))\cap C^{2,1}(\bar{\Omega}\times(0,T_{max,\varepsilon})),\\
   m_\varepsilon\in  C^0(\bar{\Omega}\times[0,T_{max,\varepsilon}))\cap C^{2,1}(\bar{\Omega}\times(0,T_{max,\varepsilon})),\\
  u_\varepsilon\in  C^0(\bar{\Omega}\times[0,T_{max,\varepsilon}))\cap C^{2,1}(\bar{\Omega}\times(0,T_{max,\varepsilon})),\\
  P_\varepsilon\in  C^{1,0}(\bar{\Omega}\times(0,T_{max,\varepsilon})),\\
   \end{array}\right.\label{1.1ddfghyuisda}
\end{equation}
 classically solving \dref{1.1fghyuisda} in $\Omega\times[0,T_{max,\varepsilon})$.
%
Moreover,  $n_\varepsilon,c_\varepsilon$ and $m_\varepsilon$ are nonnegative in
$\Omega\times(0, T_{max,\varepsilon})$, and
\begin{equation}
\|n_\varepsilon(\cdot, t)\|_{L^\infty(\Omega)}+\|c_\varepsilon(\cdot, t)\|_{W^{1,\infty}(\Omega)}+\|m_\varepsilon(\cdot, t)\|_{W^{1,\infty}(\Omega)}+\|A^\gamma u_\varepsilon(\cdot, t)\|_{L^{2}(\Omega)}\rightarrow\infty~~ \mbox{as}~~ t\nearrow T_{max,\varepsilon},
\label{1.163072x}
\end{equation}
where $\gamma$ is given by \dref{ccvvx1.731426677gg}.
\end{lemma}

\section{Some  basic estimates and  global existence in the regularized problems}
In this section we want to ensure that the time-local solutions obtained in Lemma \ref{lemma70} are in fact global
solutions. To this end, in a first step, upon a straightforward integration of the first, two and three equations in \dref{1.1fghyuisda} over $\Omega$, we can establish
the following basic estimates by using the maximum principle to the second and third equations. The detail proof
 can be found in Lemma 2.2 of \cite{EspejojjEspejojainidd793} (see also \cite{Lidfff00}).
Therefore, we list them here without proof.

%
%


\begin{lemma}\label{fvfgfflemma45}
There exists 
$\lambda > 0$ independent of $\varepsilon$ such that the solution of \dref{1.1fghyuisda} satisfies
%
%
\begin{equation}
\int_{\Omega}{n_{\varepsilon}}+\int_{\Omega}{c_{\varepsilon}}+\|m_{\varepsilon}(\cdot,t)\|_{L^\infty(\Omega)}+\|c_{\varepsilon}(\cdot,t)\|_{L^\infty(\Omega)}\leq \lambda~~\mbox{for all}~~ t\in(0, T_{max,\varepsilon})
\label{ddfgczhhhh2.5ghju48cfg924ghyuji}
\end{equation}
as well as
\begin{equation}
\|m_{\varepsilon}(\cdot,t)\|_{L^2(\Omega)}^2+2\int_0^{t}\int_{\Omega}{|\nabla m_{\varepsilon}|^{2}}\leq \lambda~~\mbox{for all}~~ t\in(0, T_{max,\varepsilon})
\label{ddczhjjjj2.5ghju48cfg9ssdd24}
\end{equation}
and
\begin{equation}
\int_0^{t}\int_{\Omega}{|\nabla c_{\varepsilon}|^{2}}\leq \lambda~~\mbox{for all}~~ t\in(0, T_{max,\varepsilon}).
\label{ddczhjjjj2.5ghju48cfgffff924}
\end{equation}
%
%
\end{lemma}
With all the above estimates at hand, we can now establish
the global existence result of our approximate solutions.

\begin{lemma}\label{kkklemmaghjmk4563025xxhjklojjkkk}
Let $\alpha\geq0$. Then
for all $\varepsilon\in(0,1),$ the solution of  \dref{1.1fghyuisda} is global in time.
\end{lemma}
\begin{proof}
{\bf Step 1: The bounded of $\|n_\varepsilon(\cdot,t)\|_{L^2(\Omega)}~~\mbox{for all}~~ t\in(0, T_{max,\varepsilon}):$}

Multiply the first equation in $\dref{1.1fghyuisda}$ by $ n_{\varepsilon}$
 and  using $\nabla\cdot u_\varepsilon=0$, we derive
\begin{equation}
\begin{array}{rl}
&\disp{\frac{1}{{2}}\frac{d}{dt}\|{ n_{\varepsilon} }\|^{{{2}}}_{L^{{2}}(\Omega)}+
\int_{\Omega}  |\nabla n_{\varepsilon}|^2}\\
=&\disp{-
\int_{\Omega}  n_{\varepsilon}\nabla\cdot(n_{\varepsilon}\frac{1}{(1+\varepsilon n_{\varepsilon})}S_\varepsilon(x, n_{\varepsilon}, c_{\varepsilon})\cdot\nabla c_{\varepsilon})-\int_{\Omega}  n_{\varepsilon}^2m_{\varepsilon}}\\
\leq&\disp{
\int_{\Omega}  n_{\varepsilon}\frac{1}{(1+\varepsilon n_{\varepsilon})}|S_\varepsilon(x, n_{\varepsilon}, c_{\varepsilon})||\nabla n_{\varepsilon}||\nabla c_{\varepsilon}|~~\mbox{for all}~~ t\in(0, T_{max,\varepsilon}),}
\end{array}
\label{55hhjjcffghhhjkklddffgglffghhhz2.5}
\end{equation}
where the last inequality we have used the nonnegativity of $n_{\varepsilon}$ and $m_{\varepsilon}$. 
Recalling \dref{x1.73142vghf48gg}, by Young inequality, one can see that
\begin{equation}
\begin{array}{rl}
&\disp\int_{\Omega} n_{\varepsilon}\frac{1}{(1+\varepsilon n_{\varepsilon})}|S_\varepsilon(x, n_{\varepsilon}, c_{\varepsilon})||\nabla n_{\varepsilon}||\nabla c_{\varepsilon}|\\
\leq&\disp{\frac{1}{\varepsilon}C_S\int_{\Omega}  |\nabla n_{\varepsilon}||\nabla c_{\varepsilon}|}\\
\leq&\disp{\frac{1}{2}\int_{\Omega}  |\nabla n_{\varepsilon}|^2+C_1\int_{\Omega} |\nabla c_{\varepsilon}|^2~~\mbox{for all}~~ t\in(0, T_{max,\varepsilon}),}
\end{array}
\label{55hhjjcffghhhjkkllfffghggggghhfghhhz2.5}
\end{equation}
where $C_1$ is a positive constant, as all subsequently appearing constants $C_2, C_3, \ldots$ possibly depend on
 $\varepsilon$.
  Substituting \dref{55hhjjcffghhhjkkllfffghggggghhfghhhz2.5} into \dref{55hhjjcffghhhjkklddffgglffghhhz2.5} and using \dref{ddczhjjjj2.5ghju48cfgffff924}, we derive that
 \begin{equation}
\begin{array}{rl}
&\disp{\int_{\Omega}n^{2}_{\varepsilon}\leq C_2~~~\mbox{for all}~~ t\in (0, T_{max,\varepsilon}).}\\
\end{array}
\label{czfvgb2.5ghhjuyucffsdffhhhhhhhjjggcvviihjj}
\end{equation}

{\bf Step 2: The bounded of $\|u_\varepsilon(\cdot,t)\|_{L^2(\Omega)}~~\mbox{for all}~~ t\in(0, T_{max,\varepsilon}):$}
Next, testing the
fourth equation of \dref{1.1fghyuisda} by $u_\varepsilon$, integrating by parts and using $\nabla\cdot u_{\varepsilon}=0$
\begin{equation}
\begin{array}{rl}
\disp{\frac{1}{2}\frac{d}{dt}\int_{\Omega}{|u_{\varepsilon}|^2}+\int_{\Omega}{|\nabla u_{\varepsilon}|^2}}= &\disp{ \int_{\Omega}(n_{\varepsilon}+m_{\varepsilon})u_{\varepsilon}\cdot\nabla \phi~~\mbox{for all}~~ t\in(0, T_{max,\varepsilon}),}\\
\end{array}
\label{ddddfgcz2.5ghju48cfg924ghyuji}
\end{equation}
where we used the facts that $\nabla \cdot u_{\varepsilon}\equiv0$ and $\nabla\cdot(1 + \varepsilon A)^{-1}u_{\varepsilon}\equiv0$.
 In light of \dref{dd1.1fghyuisdakkkllljjjkk}, \dref{ddczhjjjj2.5ghju48cfg9ssdd24} and \dref{czfvgb2.5ghhjuyucffsdffhhhhhhhjjggcvviihjj} this readily
implies
 \begin{equation}
\begin{array}{rl}
&\disp{\int_{\Omega}|u_{\varepsilon}|^{2}\leq C_3~~~\mbox{for all}~~ t\in (0, T_{max,\varepsilon})}\\
\end{array}
\label{czfvgsssssb2.5ghddhjuyucffhhhhhhhjjggcvviihjj}
\end{equation}
and some $C_3> 0.$
 Relying on properties of the Yosida approximation
$Y_{\varepsilon}$, we can also immediately find $ C_4 > 0$ and $ C_5 > 0$  such that
%
\begin{equation}
\|Y_{\varepsilon}u_{\varepsilon}\|_{L^\infty(\Omega)}=\|(I+\varepsilon A)^{-1}u_{\varepsilon}\|_{L^\infty(\Omega)}\leq C_4\|u_{\varepsilon}(\cdot,t)\|_{L^2(\Omega)}\leq C_5~~\mbox{for all}~~t\in(0,T_{max,\varepsilon}).
\label{ssdcfvgdhhjjdfghgghjjnnhhkklld911cz2.5ghju48}
\end{equation}
{\bf Step 3: The bounded of $\| u_\varepsilon(\cdot,t)\|_{L^{\infty}(\Omega)}~~\mbox{for all}~~ t\in(0, T_{max,\varepsilon}):$}
Next,  testing the projected Stokes equation $u_{\varepsilon t} +Au_{\varepsilon} =  \mathcal{P}[-\kappa (Y_{\varepsilon}u_{\varepsilon} \cdot \nabla)u_{\varepsilon}+n_{\varepsilon}\nabla \phi]$ by $Au_{\varepsilon}$, we derive
%
\begin{equation}
\begin{array}{rl}
&\disp{\frac{1}{{2}}\frac{d}{dt}\|A^{\frac{1}{2}}u_{\varepsilon}\|^{{{2}}}_{L^{{2}}(\Omega)}+
\int_{\Omega}|Au_{\varepsilon}|^2 }\\
=&\disp{ \int_{\Omega}Au_{\varepsilon}\mathcal{P}(-\kappa
(Y_{\varepsilon}u_{\varepsilon} \cdot \nabla)u_{\varepsilon})+ \int_{\Omega}\mathcal{P}[(n_{\varepsilon}+m_{\varepsilon})\nabla\phi] Au_{\varepsilon}}\\
\leq&\disp{ \frac{1}{2}\int_{\Omega}|Au_{\varepsilon}|^2+\kappa^2\int_{\Omega}
|(Y_{\varepsilon}u_{\varepsilon} \cdot \nabla)u_{\varepsilon}|^2+ \|\nabla\phi\|^2_{L^\infty(\Omega)}\int_{\Omega}(n_{\varepsilon}^2+m_{\varepsilon}^2)}\\
\leq&\disp{ \frac{1}{2}\int_{\Omega}|Au_{\varepsilon}|^2+C_6\int_{\Omega}
|\nabla u_{\varepsilon}|^2+ C_7~~\mbox{for all}~~t\in(0,T_{max,\varepsilon})}\\
\end{array}
\label{ddfghgghjjnnhhkklld911cz2.5ghju48}
\end{equation}
by using \dref{ddczhjjjj2.5ghju48cfg9ssdd24} as well as \dref{czfvgb2.5ghhjuyucffsdffhhhhhhhjjggcvviihjj} and \dref{ssdcfvgdhhjjdfghgghjjnnhhkklld911cz2.5ghju48}.
Hence,
\dref{ddfghgghjjnnhhkklld911cz2.5ghju48} implies
 \begin{equation}
\begin{array}{rl}
&\disp{\int_{\Omega}|\nabla u_{\varepsilon}|^{2}\leq C_8~~~\mbox{for all}~~ t\in (0, T_{max,\varepsilon})}\\
\end{array}
\label{czfvgsssssb2.5gcccchddhjuyucffhhhhhhhjjggcvviihjj}
\end{equation}
 by some basic calculation.
Now, let $h_{\varepsilon}(x,t)=\mathcal{P}[n_{\varepsilon}\nabla \phi-\kappa (Y_{\varepsilon}u_{\varepsilon} \cdot \nabla)u_{\varepsilon} ]$.
Then \begin{equation}
\|h_{\varepsilon}(\cdot,t)\|_{L^{2}(\Omega)}\leq C_{9} ~~~\mbox{for all}~~ t\in(0,T_{max,\varepsilon})
\label{33444cfghhh29fgsssgggx96302222114}
\end{equation}
by using \dref{czfvgb2.5ghhjuyucffsdffhhhhhhhjjggcvviihjj} and \dref{czfvgsssssb2.5gcccchddhjuyucffhhhhhhhjjggcvviihjj}. Now, we express $A^\gamma u_{\varepsilon}$ by its variation-of-constants representation and make use of well-known
smoothing properties of the Stokes semigroup (\cite{Giga1215}) to obtain $C_{10}> 0$ such that
\begin{equation}
\begin{array}{rl}
\|A^\gamma u_{\varepsilon}(\cdot, t)\|_{L^2(\Omega)}\leq&\disp{C_{10}~~ \mbox{for all}~~ t\in(0,T_{max,\varepsilon}),}\\
\end{array}
\label{cz2.57151ccvvhccvvhjjjkkhhggdddjjllll}
\end{equation}
where $\gamma\in (\frac{3}{4}, 1)$. Since,
 $D(A^\gamma)$ is continuously embedded into $L^\infty(\Omega)$ by $\gamma>\frac{3}{4}$, so that, \dref{cz2.57151ccvvhccvvhjjjkkhhggdddjjllll} yields to
  for some positive constant $C_{11}$ such that
 \begin{equation}
\begin{array}{rl}
\|u_{\varepsilon}(\cdot, t)\|_{L^\infty(\Omega)}\leq  C_{11}~~ \mbox{for all}~~ t\in(0,T_{max,\varepsilon}).\\
\end{array}
\label{cz2.5jkkcvvvhjkfffffkhhgll}
\end{equation}

{\bf Step 4: The bounded of $\| c_\varepsilon(\cdot,t)\|_{W^{1,\infty}(\Omega)}~~\mbox{for all}~~ t\in(0, T_{max,\varepsilon}):$}
Now,  test the second equation of $\dref{1.1fghyuisda}$ by $-\Delta c_{\varepsilon}$ and obtain, upon two applications of
Young¡¯s inequality, that
\begin{equation}
\begin{array}{rl}
&\disp\frac{1}{{2}}\disp\frac{d}{dt}\|\nabla{c_{\varepsilon}}\|^{{{2}}}_{L^{{2}}(\Omega)}+
\int_{\Omega} |\Delta c_{\varepsilon}|^2+ \int_{\Omega} |\nabla c_{\varepsilon}|^2\\
=&\disp{-\int_{\Omega} m_{\varepsilon}\Delta c_{\varepsilon}+\int_{\Omega} \Delta c_{\varepsilon}u_{\varepsilon}\cdot\nabla c_{\varepsilon}}\\
\leq&\disp{\frac{1}{2}\int_{\Omega} |\Delta c_{\varepsilon}|^2+\int_{\Omega} m_{\varepsilon}^2+\|u_{\varepsilon}\|^2_{L^\infty(\Omega)}\int_{\Omega} |\nabla c_{\varepsilon}|^2 .}\\
\end{array}
\label{hhxxcdfvvjjczhghhhjj2.5}
\end{equation}
 Recalling the bounds provided by  \dref{cz2.5jkkcvvvhjkfffffkhhgll} and \dref{ddczhjjjj2.5ghju48cfg9ssdd24}, this immediately implies
\begin{equation}
\int_{\Omega}{|\nabla c_{\varepsilon}(\cdot,t)|^2}\leq C_{12}~~\mbox{for all}~~ t\in(0, T_{max,\varepsilon}).
\label{ddxxxcvvdsssdddcvddffbbggddczv.5ghcfg924ghyuji}
\end{equation}
In light of Lemma 2.1 of \cite{Ishida} and the Young inequality, we have
\begin{equation}
\begin{array}{rl}
&\|\nabla c_{\varepsilon}(\cdot, t)\|_{L^{\infty}(\Omega)}\\
\leq&\disp{C_{13}(1+\|c_{\varepsilon}-u_{\varepsilon}\cdot c_{\varepsilon}\|_{L^4(\Omega)})}\\
\leq&\disp{C_{13}(1+\|c_{\varepsilon}\|_{L^4(\Omega)}+\|u_{\varepsilon}\|_{L^\infty(\Omega)} \|\nabla c_{\varepsilon}\|_{L^4(\Omega)})}\\
\leq&\disp{C_{13}(1+\|c_{\varepsilon}\|_{L^4(\Omega)}+\|u_{\varepsilon}\|_{L^\infty(\Omega)} \|\nabla c_{\varepsilon}\|_{L^\infty(\Omega)}^{\frac{1}{2}}\|\nabla c_{\varepsilon}\|_{L^2(\Omega)}^{\frac{1}{2}})}\\
\leq&\disp{C_{14}(1+ \|\nabla c_{\varepsilon}\|_{L^\infty(\Omega)}^{\frac{1}{2}})~~ \mbox{for all}~~ t\in(0,T_{max,\varepsilon}),}\\
\end{array}
\label{zjccffgbhjcvvvbscddddz2.5297x96301ku}
\end{equation}
which combined  with \dref{ddfgczhhhh2.5ghju48cfg924ghyuji} implies that
 \begin{equation}
\begin{array}{rl}
\|c_{\varepsilon}(\cdot, t)\|_{W^{1,\infty}(\Omega)}\leq  C_{15}~~ \mbox{for all}~~ t\in(0,T_{max,\varepsilon}).\\
\end{array}
\label{cz2.5jkkcvvvhjkmmffffllfkhhgll}
\end{equation}

{\bf Step 5: The bounded of $\| n_\varepsilon(\cdot,t)\|_{L^{\infty}(\Omega)}~~\mbox{for all}~~ t\in(0, T_{max,\varepsilon}):$}

Furthermore, applying the variation-of-constants formula to the $n_{\varepsilon}$-equation in \dref{1.1fghyuisda}, we get
\begin{equation}
n_{\varepsilon}(t)=e^{t\Delta}n_{\varepsilon}(\cdot,0)-\int_{0}^{t}e^{(t-s)\Delta}\nabla\cdot(n_{\varepsilon}(\cdot,s)\tilde{h}_{\varepsilon}(\cdot,s)) ds-\int_{0}^{t}e^{(t-s)\Delta}n_{\varepsilon}(\cdot,s)m_{\varepsilon}(\cdot,s)ds,~~ t\in(0,T_{max,\varepsilon}),
\label{5555fghbnmcz2.5ghjjjkkklu48cfg924ghyuji}
\end{equation}
where $\tilde{h}_{\varepsilon}:=\frac{1}{(1+\varepsilon n_{\varepsilon})}S_\varepsilon(x, n_{\varepsilon}, c_{\varepsilon})\nabla c_{\varepsilon}+u_\varepsilon$. Next, by \dref{x1.73142vghf48gg}, \dref{cz2.5jkkcvvvhjkmmffffllfkhhgll} and \dref{cz2.5jkkcvvvhjkfffffkhhgll}, we have
 $$
\begin{array}{rl}
\|\tilde{h}_{\varepsilon}(\cdot, t)\|_{L^{\infty}(\Omega)}\leq  C_{16}~~ \mbox{for all}~~ t\in(0,T_{max,\varepsilon}).\\
\end{array}
$$
As the last summand in \dref{5555fghbnmcz2.5ghjjjkkklu48cfg924ghyuji} is nonnegative by the maximum principle, so that, we can thus estimate
\begin{equation}
\begin{array}{rl}
\disp\|n_{\varepsilon}(\cdot,t)\|_{L^\infty(\Omega)}\leq&\|e^{t\Delta}n_{\varepsilon}(\cdot,0)\|_{L^\infty(\Omega)}+\disp\int_{0}^t\| e^{(t-s)\Delta}\nabla\cdot(n_{\varepsilon}(\cdot,s)\tilde{h}_{\varepsilon}(\cdot,s)\|_{L^\infty(\Omega)}ds\\
\leq&\disp\|n_{0}\|_{L^\infty(\Omega)}+\disp C_{17}\int_{0}^t(t-s)^{-\frac{7}{8}}e^{-\lambda_1(t-s)}\|n_{\varepsilon}(\cdot,s)\tilde{h}_{\varepsilon}(\cdot,s)\|_{L^4(\Omega)}ds\\
\leq&\disp\|n_{0}\|_{L^\infty(\Omega)}+\disp C_{18}\int_{0}^t(t-s)^{-\frac{7}{8}}e^{-\lambda_1(t-s)}\|n_{\varepsilon}(\cdot,s)\|_{L^4(\Omega)}ds\\
\leq&\disp\|n_{0}\|_{L^\infty(\Omega)}+\disp C_{18}\int_{0}^t(t-s)^{-\frac{7}{8}}e^{-\lambda_1(t-s)}\|n_{\varepsilon}(\cdot,s)\|_{L^\infty(\Omega)}^{\frac{1}{2}}
\|n_{\varepsilon}(\cdot,s)\|_{L^2(\Omega)}^{\frac{1}{2}}ds\\
\leq&\disp\|n_{0}\|_{L^\infty(\Omega)}+\disp C_{19}\sup_{s\in(0, T_{max,\varepsilon})}\|n_{\varepsilon}(\cdot,s)\|_{L^\infty(\Omega)}^{\frac{1}{2}}~~\mbox{for all}~~ t\in(0,T_{max,\varepsilon}),\\
\end{array}
\label{ccvbccvvbbnnndffghhjjvcvvbcsssscfbbnfgbghjjccmmllffvvggcvvvvbbjjkkdffzjscz2.5297x9630xxy}
\end{equation}
where
$$\begin{array}{rl}\disp C_{18}=\int_{0}^{\infty}s^{-\frac{7}{8}}e^{-\lambda_1s}\|n_{\varepsilon}(\cdot,s)\|_{L^2(\Omega)}^{\frac{1}{2}}
&\disp{<+\infty,}\\
\end{array}
$$
$\lambda_1$ is the first nonzero eigenvalue of $-\Delta$ on $\Omega$ under
the Neumann boundary condition.
And  thereby
\begin{equation}
\begin{array}{rl}
\|n_{\varepsilon}(\cdot, t)\|_{L^{\infty}(\Omega)}\leq&\disp{C_{19}~~ \mbox{for all}~~ t\in(0,T_{max,\varepsilon})}\\
\end{array}
\label{cz2ddff.57151ccvsssshhjjjkkkuuifghhhivhccvvhjjjkkhhggjjllll}
\end{equation}
by using the Young inequality.
%

Assume that $T_{max,\varepsilon}< \infty$.
In view of  \dref{cz2.57151ccvvhccvvhjjjkkhhggdddjjllll},
\dref{cz2.5jkkcvvvhjkmmffffllfkhhgll} and \dref{cz2ddff.57151ccvsssshhjjjkkkuuifghhhivhccvvhjjjkkhhggjjllll}, we apply Lemma \ref{lemma70} to reach a contradiction.
\end{proof}

\section{A priori estimates for the regularized problem \dref{1.1fghyuisda} which is independent of $\varepsilon$ }
Since we want to obtain a weak solution of \dref{1.1} by means of taking $\varepsilon\searrow 0$ in \dref{1.1fghyuisda}, we will require
regularity information which is independent of $\varepsilon\in (0,1)$. The main portion of important estimates will be
prepared in the following section.

\begin{lemma}\label{lemmaghjffggssddgghhmk4563025xxhjklojjkkk}
Let $\alpha>0$ and $p=4\alpha+\frac{2}{3}$.
Then there exists $C>0$ independent of $\varepsilon$ such that the solution of \dref{1.1fghyuisda} satisfies
\begin{equation}
\begin{array}{rl}
&\disp{\int_{\Omega} n_{\varepsilon}^{p}+\int_{\Omega}   |\nabla c_{\varepsilon}|^2+\int_{\Omega}  | {u_{\varepsilon}}|^2\leq C~~~\mbox{for all}~~ t\in (0, T_{max,\varepsilon}).}\\
\end{array}
\label{czfvgb2.5ghhjuyuccvviihjj}
\end{equation}
Moreover, for $T\in(0, T_{max,\varepsilon})$, it holds that
one can find a constant $C > 0$ independent of $\varepsilon$ such that
\begin{equation}
\begin{array}{rl}
&\disp{\int_{0}^T\left[\int_{\Omega}  |\nabla {u_{\varepsilon}}|^2+\int_{\Omega}|\nabla c_{\varepsilon}|^4 +\|\nabla n_{\varepsilon}^{\frac{p}{2}}\|^2_{L^2(\Omega)}+\int_{\Omega} |\Delta c_{\varepsilon}|^2\right]\leq C(T+1).}\\
\end{array}
\label{bnmbncz2.5ghhjuyuivvbnnihjj}
\end{equation}
\end{lemma}
\begin{proof}
Let $p=4\alpha+\frac{2}{3}$. 
%
We first obtain from $\nabla\cdot u_\varepsilon=0$ in
 $\Omega\times (0, T_{max,\varepsilon})$ and straightforward calculations
that
\begin{equation}
\begin{array}{rl}
&\disp{sign(p-1)\frac{1}{{p }}\frac{d}{dt}\| n_{\varepsilon} \|^{{{p }}}_{L^{{p }}(\Omega)}+sign(p-1)
 \frac{4({p}-1)}{p^2}\|n_{\varepsilon}^{\frac{p}{2}}\|^2_{L^2(\Omega)}}\\
=&\disp{-sign(p-1)
\int_{\Omega}   n_{\varepsilon}^{{p}-1}\nabla\cdot(n_{\varepsilon}\frac{1}{(1+\varepsilon n_{\varepsilon})}S_\varepsilon(x, n_{\varepsilon}, c_{\varepsilon})\cdot\nabla c_{\varepsilon}-  sign(p-1)  \int_{\Omega}n_{\varepsilon}^{{p}-1}m_{\varepsilon}}\\
\leq&\disp{ sign(p-1)({p}-1)
\int_{\Omega}  n_{\varepsilon}^{{p}-1}\frac{1}{(1+\varepsilon n_{\varepsilon})}|S_\varepsilon(x, n_{\varepsilon}, c_{\varepsilon})||\nabla n_{\varepsilon}||\nabla c_{\varepsilon}|- sign(p-1)  \int_{\Omega}n_{\varepsilon}^{{p}-1}m_{\varepsilon}}
\end{array}
\label{55hhjjcffghhhjkkllz2.5}
\end{equation}
for all $t\in(0, T_{max,\varepsilon}).$
  Therefore,  
 in light of \dref{x1.73142vghf48gg}, with the help of  the Young inequality, we can estimate the right of \dref{55hhjjcffghhhjkkllz2.5} by following
\begin{equation}
\begin{array}{rl}
&\disp{sign(p-1) ({p}-1)
\int_{\Omega}  n_{\varepsilon}^{{p}-1}\frac{1}{(1+\varepsilon n_{\varepsilon})}|S_\varepsilon(x, n_{\varepsilon}, c_{\varepsilon})||\nabla n_{\varepsilon}||\nabla c_{\varepsilon}|}\\
\leq&\disp{ |{p}-1|
\int_{\Omega}  n_{\varepsilon}^{{p}-1}C_S(1 + n_{\varepsilon})^{-\alpha}|\nabla n_{\varepsilon}||\nabla c_{\varepsilon}|}\\
\leq&\disp{ \frac{|{p}-1|}{2}\int_{\Omega}  n_{\varepsilon}^{{p}-2} |\nabla n_{\varepsilon}|^2+\frac{|{p}-1|}{2}C_S^2\int_{\Omega}  n_{\varepsilon}^{{p}}(1 + n_{\varepsilon})^{-2\alpha}|\nabla c_{\varepsilon}|^2}\\
\leq&\disp{ \frac{|{p}-1|}{2}\int_{\Omega}  n_{\varepsilon}^{{p}-2} |\nabla n_{\varepsilon}|^2+\frac{|{p}-1|}{2}C_S^2\int_{\Omega}  n_{\varepsilon}^{{p}-2\alpha}|\nabla c_{\varepsilon}|^2}\\
=&\disp{ \frac{2|{p}-1|}{p^2}\|n_{\varepsilon}^{\frac{p}{2}}\|^2_{L^2(\Omega)}+\frac{|{p}-1|}{2}C_S^2\int_{\Omega}  n_{\varepsilon}^{{p}-2\alpha}|\nabla c_{\varepsilon}|^2~~\mbox{for all}~~ t\in(0, T_{max,\varepsilon})}\\
\end{array}
\label{55hhjjcz2.5}
\end{equation}
by using  the fact that
$  (1 + n_{\varepsilon})^{-2\alpha}\leq n_{\varepsilon}^{-2\alpha}$ for all
$\varepsilon\geq0,$
$n_{\varepsilon}$ and $\alpha\geq0$. In the following we will estimate the term $\frac{|p-1|}{2}C_S^2\int_{\Omega}  n_{\varepsilon}^{{p}-2\alpha}|\nabla c_{\varepsilon}|^2$
in the right hand side of \dref{55hhjjcz2.5}.  To this end,
we firstly  invoke the Gagliardo-Nirenberg inequality again to obtain $C_1> 0$ and $C_2>0$ such
that
\begin{equation}
\begin{array}{rl}
 \disp\int_{\Omega}  n_{\varepsilon}^{2{p}-4\alpha}=&\disp{ \|n_{\varepsilon}^{\frac{p}{2}}\|^{\frac{4(p-2\alpha)}{p}}_{L^\frac{4(p-2\alpha)}{p}(\Omega)}}\\
 \leq&\disp{C_1 \|\nabla n_{\varepsilon}^{\frac{p}{2}}\|^{\frac{2(6p-12\alpha-3)}{3p-1}}_{L^2(\Omega)}\| n_{\varepsilon}^{\frac{p}{2}}\|^{\frac{4(p-2\alpha)}{p}-\frac{2(6p-12\alpha-3)}{3p-1}}_{L^{\frac{2}{p}}(\Omega)}+C_1 \|n_{\varepsilon}^{\frac{p}{2}}\|^{\frac{4(p-2\alpha)}{p}}_{L^\frac{2}{p}(\Omega)}}\\
  \leq&\disp{C_2( \|\nabla n_{\varepsilon}^{\frac{p}{2}}\|^{\frac{2(6p-12\alpha-3)}{3p-1}}_{L^2(\Omega)}+1)}\\
  =&\disp{C_2( \|\nabla n_{\varepsilon}^{\frac{p}{2}}\|^{2}_{L^2(\Omega)}+1)}\\
\end{array}
\label{55hherrrrjjssssssddddcz2.5}
\end{equation}
by using \dref{ddfgczhhhh2.5ghju48cfg924ghyuji} and $p=4\alpha+\frac{2}{3}$.
Next,
recalling the Young inequality, 
\begin{equation}
\begin{array}{rl}
&\disp{ \frac{|p-1|}{2}C_S^2\int_{\Omega}  n_{\varepsilon}^{{p}-2\alpha}|\nabla c_{\varepsilon}|^2\leq\frac{1}{C_2}\frac{|{p}-1|}{2p^2}\int_{\Omega}  n_{\varepsilon}^{2{p}-4\alpha} + C_3\int_{\Omega}|\nabla c_{\varepsilon}|^4~~\mbox{for all}~~ t\in(0, T_{max,\varepsilon})}\\
\end{array}
\label{55hherrrrjjssdcz2.5}
\end{equation}
with $C_3=\disp\frac{p^2C_2|p-1|C_S^4}{8},$ where $C_2$ is the same as \dref{55hherrrrjjssssssddddcz2.5}.
Inserting  \dref{55hherrrrjjssdcz2.5} into \dref{55hhjjcz2.5}, we may derive that
\begin{equation}
\begin{array}{rl}
&\disp{sign(p-1) ({p}-1)
\int_{\Omega}  n_{\varepsilon}^{{p}-1}\frac{1}{(1+\varepsilon n_{\varepsilon})}|S_\varepsilon(x, n_{\varepsilon}, c_{\varepsilon})||\nabla n_{\varepsilon}||\nabla c_{\varepsilon}|}\\
\leq&\disp{  \frac{2|{p}-1|}{p^2}\|n_{\varepsilon}^{\frac{p}{2}}\|^2_{L^2(\Omega)}+\frac{1}{C_2}\frac{|{p}-1|}{2p^2}\int_{\Omega}  n_{\varepsilon}^{2{p}-4\alpha} + C_3\int_{\Omega}|\nabla c_{\varepsilon}|^4~~\mbox{for all}~~ t\in(0, T_{max,\varepsilon}).}\\
\end{array}
\label{55hhjjcz2.ssdfgh5}
\end{equation}
Next, using the Gagliardo-Nirenberg inequality and \dref{ddfgczhhhh2.5ghju48cfg924ghyuji}, one can get
\begin{equation}
\begin{array}{rl}
\disp \int_{\Omega}|\nabla c_{\varepsilon}|^4=\|\nabla c_{\varepsilon}\|_{L^{4}(\Omega)}^4\leq&\disp{\lambda_0\|\Delta c_{\varepsilon}\|_{L^{2}(\Omega)}^2\| c_{\varepsilon}\|_{L^{\infty}(\Omega)}^2+\lambda_0\|c_{\varepsilon}\|_{L^{\infty}(\Omega)}^4}\\
\leq&\disp{\lambda_2\|\Delta c_{\varepsilon}\|_{L^{2}(\Omega)}^2+\lambda_1~~\mbox{for all}~~ t\in(0, T_{max,\varepsilon})}\\
\end{array}
\label{hhxxcdfvvdfghhhsssffjjddddcz2.5}
\end{equation}
for some positive constants $\lambda_0,\lambda_1$ and $\lambda_2$ independent of $\varepsilon.$
%
Collecting \dref{55hhjjcffghhhjkkllz2.5}--\dref{55hherrrrjjssssssddddcz2.5} and
\dref{55hhjjcz2.ssdfgh5}--\dref{hhxxcdfvvdfghhhsssffjjddddcz2.5}, we conclude that there exist positive constants $C_4$ and $C_5$ such that
\begin{equation}
\begin{array}{rl}
&\disp{ sign(p-1)\frac{1}{{p }}\frac{d}{dt}\| n_{\varepsilon} \|^{{{p }}}_{L^{{p }}(\Omega)}+
 (sign(p-1)\frac{4({p}-1)}{p^2}-\frac{5|{p}-1|}{2p^2})\|\nabla n_{\varepsilon}^{\frac{p}{2}}\|^2_{L^2(\Omega)}}\\
 \leq&\disp{ C_4\|\Delta c_{\varepsilon}\|_{L^{2}(\Omega)}^2+C_5- sign(p-1)  \int_{\Omega}n_{\varepsilon}^{{p}-1}m_{\varepsilon}~~\mbox{for all}~~ t\in(0, T_{max,\varepsilon}).}
\end{array}
\label{55hhjjcffghhhjkkllz2dddff.5}
\end{equation}

To track the time evolution of $c_\varepsilon$,
taking ${-\Delta c_{\varepsilon}}$ as the test function for the second  equation of \dref{1.1fghyuisda} and using \dref{ddfgczhhhh2.5ghju48cfg924ghyuji}, we have 
\begin{equation}
\begin{array}{rl}
&\disp\frac{1}{{2}}\disp\frac{d}{dt}\|\nabla{c_{\varepsilon}}\|^{{{2}}}_{L^{{2}}(\Omega)}+
\int_{\Omega} |\Delta c_{\varepsilon}|^2+ \int_{\Omega} |\nabla c_{\varepsilon}|^2\\
=&\disp{-\int_{\Omega} m_{\varepsilon}\Delta c_{\varepsilon}+\int_{\Omega} \Delta c_{\varepsilon}u_{\varepsilon}\cdot\nabla c_{\varepsilon}}\\
=&\disp{-\int_{\Omega} m_{\varepsilon}\Delta c_{\varepsilon}-\int_{\Omega} \nabla c_{\varepsilon}\cdot\nabla ( u_{\varepsilon}\cdot\nabla c_{\varepsilon})}\\
=&\disp{-\int_{\Omega} m_{\varepsilon}\Delta c_{\varepsilon}-\int_{\Omega} \nabla c_{\varepsilon} \cdot(\nabla u_{\varepsilon}\cdot\nabla c_{\varepsilon})-\int_{\Omega} \nabla c_{\varepsilon} \cdot(D^2 \cdot u_{\varepsilon}),}\\
\end{array}
\label{hhxxcdfvvjjcz2.5}
\end{equation}
which together with the fact that
$$-\int_{\Omega} \nabla c_{\varepsilon} \cdot(D^2 \cdot u_{\varepsilon})=-\frac{1}{2}\int_{\Omega} u_{\varepsilon} \cdot\nabla|\nabla c_{\varepsilon}|^2=0$$
implies that
\begin{equation}
\begin{array}{rl}
&\disp\frac{1}{{2}}\disp\frac{d}{dt}\|\nabla{c_{\varepsilon}}\|^{{{2}}}_{L^{{2}}(\Omega)}+
\int_{\Omega} |\Delta c_{\varepsilon}|^2+ \int_{\Omega} |\nabla c_{\varepsilon}|^2\\
\leq&\disp{\frac{1}{4}\int_{\Omega} |\Delta c_{\varepsilon}|^2+\int_{\Omega} |m_{\varepsilon}|^2+\|\nabla c_{\varepsilon}\|_{L^4(\Omega)}^2 \|\nabla u_{\varepsilon}\|_{L^2(\Omega)}}\\
\leq&\disp{\frac{1}{4}\int_{\Omega} |\Delta c_{\varepsilon}|^2+\int_{\Omega} |m_{\varepsilon}|^2+\frac{1}{4\lambda_2}\|\nabla c_{\varepsilon}\|_{L^4(\Omega)}^4 +\lambda_2\|\nabla u_{\varepsilon}\|_{L^2(\Omega)}^2,}\\
\end{array}
\label{hhxxcdfvvjjczdddf2.5}
\end{equation}
where $\lambda_2$ is the same as \dref{hhxxcdfvvdfghhhsssffjjddddcz2.5}.
This together with \dref{hhxxcdfvvdfghhhsssffjjddddcz2.5} yields to
\begin{equation}
\begin{array}{rl}
\disp\frac{1}{{2}}\disp\frac{d}{dt}\|\nabla{c_{\varepsilon}}\|^{{{2}}}_{L^{{2}}(\Omega)}+
\frac{1}{2}\int_{\Omega} |\Delta c_{\varepsilon}|^2+ \int_{\Omega} |\nabla c_{\varepsilon}|^2\leq&\disp{\lambda_2\|\nabla u_{\varepsilon}\|_{L^2(\Omega)}^2+C_6}\\
\end{array}
\label{hhxxcdfvvjjcssfggzdddf2.5}
\end{equation}
by using \dref{ddczhjjjj2.5ghju48cfg9ssdd24}.
Take an evident linear combination of the inequalities provided by \dref{55hhjjcffghhhjkkllz2dddff.5} and \dref{hhxxcdfvvjjcssfggzdddf2.5}, we conclude
\begin{equation}
\begin{array}{rl}
&\disp{sign(p-1) \frac{1}{{p }}\frac{d}{dt}\| n_{\varepsilon} \|^{{{p }}}_{L^{{p }}(\Omega)}+2C_4\disp\frac{d}{dt}\|\nabla{c_{\varepsilon}}\|^{{{2}}}_{L^{{2}}(\Omega)}}\\
&\disp{+(sign(p-1)\frac{4({p}-1)}{p^2}-\frac{5|{p}-1|}{2p^2})\|\nabla n_{\varepsilon}^{\frac{p}{2}}\|^2_{L^2(\Omega)}}\\
 &+\disp{C_4\int_{\Omega} |\Delta c_{\varepsilon}|^2+ 2C_4\int_{\Omega} |\nabla c_{\varepsilon}|^2}\\
 \leq &{\kappa_0\|\nabla u_{\varepsilon}\|_{L^2(\Omega)}^2+C_7- sign(p-1)  \disp\int_{\Omega}n_{\varepsilon}^{{p}-1}m_{\varepsilon}~~\mbox{for all}~~ t\in(0, T_{max,\varepsilon}),}
\end{array}
\label{55hhjjcffghhhjkkllzssss2dddffssddd.5}
\end{equation}
where \begin{equation}
\kappa_0=2\lambda_2C_4,C_7=2C_6C_4+C_5.
\label{55hhjjcffghhhjkksssssdddllzssss2dddfssddfssddd.5}
\end{equation}
%

Now, multiplying the
fourth equation of \dref{1.1fghyuisda} by $u_\varepsilon$, integrating by parts and using $\nabla\cdot u_{\varepsilon}=0$
\begin{equation}
\begin{array}{rl}
\disp{\frac{1}{2}\frac{d}{dt}\int_{\Omega}{|u_{\varepsilon}|^2}+\int_{\Omega}{|\nabla u_{\varepsilon}|^2}}= &\disp{ \int_{\Omega}(n_{\varepsilon}+m_{\varepsilon})u_{\varepsilon}\cdot\nabla \phi~~\mbox{for all}~~ t\in(0, T_{max,\varepsilon}).}\\
\end{array}
\label{ddddfgcz2.5ghju48cfg924ghyuji}
\end{equation}
Noticing the fact $W^{1,2}(\Omega)\hookrightarrow L^6(\Omega)$  in the 3D case and making
use of the H\"{o}lder inequality  and the Young inequality we can estimate the second term in
the right hand of \dref{ddddfgcz2.5ghju48cfg924ghyuji} as
\begin{equation}
\begin{array}{rl}
\disp\int_{\Omega}(n_{\varepsilon}+m_{\varepsilon})u_{\varepsilon}\cdot\nabla \phi\leq&\disp{\|\nabla \phi\|_{L^\infty(\Omega)}\| n_{\varepsilon} \|_{L^{\frac{6}{5}}(\Omega)}\| u_{\varepsilon}\|_{L^{6}(\Omega)}+\|\nabla \phi\|_{L^\infty(\Omega)}\| m_{\varepsilon} \|_{L^{\frac{6}{5}}(\Omega)}\| u_{\varepsilon}\|_{L^{6}(\Omega)}}\\
\leq&\disp{C_{8}\|\nabla \phi\|_{L^\infty(\Omega)}(\| n_{\varepsilon} \|_{L^{\frac{6}{5}}(\Omega)}+\|m_{\varepsilon} \|_{L^{\frac{6}{5}}(\Omega)})\|\nabla u_{\varepsilon}\|_{L^{2}(\Omega)}}\\
\leq&\disp{\frac{1}{2}\|\nabla u_{\varepsilon}\|_{L^{2}(\Omega)}^2+C_{9}(\| n_{\varepsilon} \|_{L^{\frac{6}{5}}(\Omega)}^2+1)~~\mbox{for all}~~ t\in(0, T_{max,\varepsilon})}\\
\end{array}
\label{dddddddddfgcz2.5ghju48cfg924ghyuji}
\end{equation}
by using \dref{ddfgczhhhh2.5ghju48cfg924ghyuji} and \dref{dd1.1fghyuisdakkkllljjjkk}.
Next,  the Young inequality along with the assumed boundedness of $\phi$ (see \dref{dd1.1fghyuisdakkkllljjjkk})  as well as the Gagliardo--Nirenberg inequality and \dref{ddfgczhhhh2.5ghju48cfg924ghyuji}  yields
%
\begin{equation}
\begin{array}{rl}
\disp{\int_{\Omega}(n_{\varepsilon}+m_{\varepsilon})u_{\varepsilon}\cdot\nabla \phi}\leq&\disp{\frac{1}{2}\|\nabla u_{\varepsilon}\|_{L^{2}(\Omega)}^2+C_{9}
\| n_{\varepsilon}^{\frac{p}{2} } \|_{L^{\frac{12}{5p}}(\Omega)}^{\frac{4}{p}}+C_{9}}\\
\leq &\disp{\frac{1}{2}\|\nabla u_{\varepsilon}\|_{L^{2}(\Omega)}^2+C_{10}\|\nabla n_{\varepsilon}^{\frac{p}{2} }\|_{L^{2}(\Omega)}^{\frac{2}{3p-1}}\|  n_{\varepsilon}^{\frac{p}{2} }\|_{L^{\frac{2}{p }}(\Omega)}^{\frac{4}{p}-\frac{2}{3p-1}}+C_{10}\| n_{\varepsilon}^{\frac{p}{2} } \|_{L^{\frac{2}{p}}(\Omega)}^{\frac{4}{p}}+C_{9}}\\
\leq &\disp{\frac{1}{2}\|\nabla u_{\varepsilon}\|_{L^{2}(\Omega)}^2+\frac{|{p}-1|}{2p^2}\frac{1}{4\kappa_0}\|\nabla  n_{\varepsilon}^{\frac{p}{2} }\|_{L^{2}(\Omega)}^{2}+C_{11}~~\mbox{for all}~~ t\in(0, T_{max,\varepsilon}),}\\
\end{array}
\label{ddddfgcxccdd2.5ghju4cvvbbttthdfff8cfg924ghyuji}
\end{equation}
where $\kappa_0$ is given by \dref{55hhjjcffghhhjkksssssdddllzssss2dddfssddfssddd.5}, $C_{9}, C_{10}$ and $C_{11}$ are positive constants which are independent of $\varepsilon.$
Here the last inequality we have used the fact that
$$\frac{2}{3p-1}>2~~\mbox{by}~~p=4\alpha+\frac{2}{3}>\frac{2}{3}.$$
Now, substituting  \dref{ddddfgcxccdd2.5ghju4cvvbbttthdfff8cfg924ghyuji} into \dref{ddddfgcz2.5ghju48cfg924ghyuji}, one has
\begin{equation}
\begin{array}{rl}
\disp{\frac{1}{2}\frac{d}{dt}\int_{\Omega}{|u_{\varepsilon}|^2}+\frac{1}{2}\|\nabla u_{\varepsilon}\|_{L^{2}(\Omega)}^2}\leq &\disp{ \frac{|{p}-1|}{2p^2}\frac{1}{4\kappa_0}\|\nabla  n_{\varepsilon}^{\frac{p}{2} }\|_{L^{2}(\Omega)}^{2}+C_{12}~~\mbox{for all}~~ t\in(0, T_{max,\varepsilon}),}\\
\end{array}
\label{ddddfgcz2.5ghju48cfg924ssddddghyuji}
\end{equation}
so that,  which together with \dref{55hhjjcffghhhjkkllzssss2dddffssddd.5} implies that
\begin{equation}
\begin{array}{rl}
&\disp{2\kappa_0\frac{d}{dt}\int_{\Omega}{|u_{\varepsilon}|^2}+sign(p-1)\frac{1}{{p }}\frac{d}{dt}\| n_{\varepsilon} \|^{{{p }}}_{L^{{p }}(\Omega)}+2C_4\disp\frac{d}{dt}\|\nabla{c_{\varepsilon}}\|^{{{2}}}_{L^{{2}}(\Omega)}}\\
&\disp{+\kappa_0\|\nabla u_{\varepsilon}\|_{L^{2}(\Omega)}^2+(sign(p-1)\frac{4({p}-1)}{p^2}-\frac{3|{p}-1|}{p^2})\|\nabla n_{\varepsilon}^{\frac{p}{2}}\|^2_{L^2(\Omega)}}\\
&+\disp{C_4\int_{\Omega} |\Delta c_{\varepsilon}|^2+ 2C_4\int_{\Omega} |\nabla c_{\varepsilon}|^2}\\
\leq &\disp{ C_{13}- sign(p-1)  \disp\int_{\Omega}n_{\varepsilon}^{{p}-1}m_{\varepsilon}~~\mbox{for all}~~ t\in(0, T_{max,\varepsilon}).}\\
\end{array}
\label{sdfffddddfgcz2.5ghjgghussdd48cfg924ssddddghyuji}
\end{equation}
 Case $p>1:$ Then $sign(p-1)=1>0.$
Thus, \dref{sdfffddddfgcz2.5ghjgghussdd48cfg924ssddddghyuji} implies that
\begin{equation}
\begin{array}{rl}
&\disp{2\kappa_0\frac{d}{dt}\int_{\Omega}{|u_{\varepsilon}|^2}+\frac{1}{{p }}\frac{d}{dt}\| n_{\varepsilon} \|^{{{p }}}_{L^{{p }}(\Omega)}+2C_4\disp\frac{d}{dt}\|\nabla{c_{\varepsilon}}\|^{{{2}}}_{L^{{2}}(\Omega)}}\\
&\disp{+\kappa_0\|\nabla u_{\varepsilon}\|_{L^{2}(\Omega)}^2+\frac{|{p}-1|}{p^2}\|\nabla n_{\varepsilon}^{\frac{p}{2}}\|^2_{L^2(\Omega)}}\\
&+\disp{C_4\int_{\Omega} |\Delta c_{\varepsilon}|^2+ 2C_4\int_{\Omega} |\nabla c_{\varepsilon}|^2}\\
\leq &\disp{ C_{13}~~\mbox{for all}~~ t\in(0, T_{max,\varepsilon})}\\
\end{array}
\label{ddddfgcz2.5ghjgghussdd48cfg924ssddddghyuji}
\end{equation}
by using
$$- sign(p-1)  \disp\int_{\Omega}n_{\varepsilon}^{{p}-1}m_{\varepsilon}=-  \disp\int_{\Omega}n_{\varepsilon}^{{p}-1}m_{\varepsilon}\leq0,$$
Next,
 integrating \dref{ddddfgcz2.5ghjgghussdd48cfg924ssddddghyuji} in time, we can obtain from \dref{hhxxcdfvvdfghhhsssffjjddddcz2.5} that
\begin{equation}
\begin{array}{rl}
&\disp{\int_{\Omega}   |u_{\varepsilon}|^2+\int_{\Omega}   n_{\varepsilon}^p+\int_{\Omega}   |\nabla c_{\varepsilon}|^2\leq C_{14}~~~\mbox{for all}~~ t\in (0, T_{max,\varepsilon})}\\
\end{array}
\label{czfvgb2.5ghhjuyghjjjuddfghhccvjkkklllhhjkkviihjj}
\end{equation}
and
\begin{equation}
\begin{array}{rl}
&\disp{\int_{0}^T\left[\int_{\Omega}  |\nabla {u_{\varepsilon}}|^2+\int_{\Omega}|\nabla c_{\varepsilon}|^4 +\|\nabla n_{\varepsilon}^{\frac{p}{2}}\|^2_{L^2(\Omega)}+\int_{\Omega} |\Delta c_{\varepsilon}|^2\right]\leq C_{14}(T+1)~~\mbox{for all}~~ T\in(0, T_{max,\varepsilon})}\\
\end{array}
\label{bnmbncz2.5ghhjuddfghhffghhhjjkklhddfggyhjkklluivvbnnihjj}
\end{equation}
and some positive constant $C_{14}.$
Case $p=4\alpha+\frac{2}{3}<1$:
 Then $ sign(p-1)=-1<0$, hence, in view of \dref{ddfgczhhhh2.5ghju48cfg924ghyuji},
 integrating \dref{sdfffddddfgcz2.5ghjgghussdd48cfg924ssddddghyuji} in time and employing  the H\"{o}lder inequality, we also conclude that there exists a positive constant $C_{15}$ such that
 \begin{equation}
\begin{array}{rl}
&\disp{\int_{\Omega}   |u_{\varepsilon}|^2+\int_{\Omega}   n_{\varepsilon}^p+\int_{\Omega}   |\nabla c_{\varepsilon}|^2\leq C_{15}~~~\mbox{for all}~~ t\in (0, T_{max,\varepsilon})}\\
\end{array}
\label{czfvgb2.5ghhjuyuddfghhccvhhjkkviihjj}
\end{equation}
and
\begin{equation}
\begin{array}{rl}
&\disp{\int_{0}^T\left[\int_{\Omega}  |\nabla {u_{\varepsilon}}|^2+\int_{\Omega}|\nabla c_{\varepsilon}|^4 +\|\nabla n_{\varepsilon}^{\frac{p}{2}}\|^2_{L^2(\Omega)}+\int_{\Omega} |\Delta c_{\varepsilon}|^2\right]\leq C_{15}(T+1)~~\mbox{for all}~~ T\in(0, T_{max,\varepsilon}).}\\
\end{array}
\label{bnmbncz2.5ghhjuddfghhddfggyhjkklluivvbnnihjj}
\end{equation}
Case£º $p=1:$
Using the first equation of \dref{1.1fghyuisda}, from integration by parts and applying \dref{x1.73142vghf48gg}, we derive from \dref{ddfgczhhhh2.5ghju48cfg924ghyuji} that
\begin{equation}
 \begin{array}{rl}
&\disp\frac{d}{dt}\disp\int_{\Omega} n_{\varepsilon} \ln  n_{\varepsilon}\\
 =&\disp{\int_{\Omega}n_{\varepsilon t} \ln  n_{\varepsilon}+
\int_{\Omega}n_{\varepsilon t}}\\
=&\disp{\int_{\Omega}\Delta  n_{\varepsilon}  \ln  n_{\varepsilon}-
\int_{\Omega}\ln  n_{\varepsilon}\nabla\cdot(n_{\varepsilon}\frac{1}{(1+\varepsilon n_{\varepsilon})}S_\varepsilon(x, n_{\varepsilon}, c_{\varepsilon})
\cdot\nabla c_{\varepsilon})-  \int_{\Omega}\ln  n_{\varepsilon}m_{\varepsilon}-\int_{\Omega}  n_{\varepsilon}m_{\varepsilon}}\\
\leq&\disp{-\int_{\Omega} \frac{|\nabla n_{\varepsilon}|^2}{n_{\varepsilon}}+
\int_{\Omega} C_S(1 + n_{\varepsilon})^{-\alpha}\frac{n_{\varepsilon}}{ n_{\varepsilon} }|\nabla n_{\varepsilon}||\nabla c_{\varepsilon}|+C_{16}~~~\mbox{for all}~~ t\in (0, T_{max,\varepsilon}),}\\
\end{array}\label{vgccvsckkcvvsbbbbhsvvbbsddaqwswddaassffssff3.10deerfgghhjuuloollgghhhyhh}
\end{equation}
which combined with the Young inequality implies that
\begin{equation}
 \begin{array}{rl}
\disp\frac{d}{dt}\disp\int_{\Omega} n_{\varepsilon} \ln  n_{\varepsilon} +\disp\frac{1}{2}
\disp\int_{\Omega} \frac{|\nabla n_{\varepsilon}|^2}{n_{\varepsilon}} \leq&\disp{\disp\frac{1}{2}C_S^2
\disp\int_{\Omega}\frac{n_{\varepsilon}|\nabla c_{\varepsilon}|^2}{(1 + n_{\varepsilon})^{2\alpha}}}\\
\leq&\disp{\disp\frac{1}{2}C_S^2
\disp\int_{\Omega}n_{\varepsilon}^{1-2\alpha}|\nabla c_{\varepsilon}|^2+C_{16}~~~\mbox{for all}~~ t\in (0, T_{max,\varepsilon}).}\\
\end{array}\label{vgccvsckkcvvsbbbbhsvvbbsddaqwswddaassffssff3.10defgghjjerfgghhjuuloollgghhhyhh}
\end{equation}
On the  other hand, due to  $p=1$  yields to $4\alpha+\frac{2}{3}>\frac{2}{3}$, employing almost exactly the same arguments as in the proof of \dref{hhxxcdfvvjjcz2.5}--\dref{bnmbncz2.5ghhjuddfghhffghhhjjkklhddfggyhjkklluivvbnnihjj} (the minor necessary changes
are left as an easy exercise to the reader),  we conclude the estimate
\begin{equation}
\begin{array}{rl}
&\disp{\int_{\Omega}   |u_{\varepsilon}|^2+\int_{\Omega}   n_{\varepsilon}^p+\int_{\Omega}   |\nabla c_{\varepsilon}|^2\leq C_{17}~~~\mbox{for all}~~ t\in (0, T_{max,\varepsilon})}\\
\end{array}
\label{czfvgb2.5ghhjuyuddfghhccvhjkkkkhhjkkklllkviihjj}
\end{equation}
and
\begin{equation}
\begin{array}{rl}
&\disp{\int_{0}^T\left[\int_{\Omega}  |\nabla {u_{\varepsilon}}|^2+\int_{\Omega}|\nabla c_{\varepsilon}|^4 +\|\nabla n_{\varepsilon}^{\frac{p}{2}}\|^2_{L^2(\Omega)}+\int_{\Omega} |\Delta c_{\varepsilon}|^2\right]\leq C_{17}(T+1)~~\mbox{for all}~~ T\in(0, T_{max,\varepsilon}).}\\
\end{array}
\label{bnmbncz2.5ghhjuddfghhffghhddfggyhjkkllujjkkivvbnnihjj}
\end{equation}
\end{proof}


\subsection{Further a-priori estimates}

With the help of Lemma \ref{lemmaghjffggssddgghhmk4563025xxhjklojjkkk} and  the Gagliardo--Nirenberg inequality,
we can derive the following Lemma:
\begin{lemma}\label{lemmddaghjsffggggsddgghhmk4563025xxhjklojjkkk}
Let $\alpha>0$. Then for each $T\in(0, T_{max,\varepsilon})$,
 there exists $C>0$ independent of $\varepsilon$ such that the solution of \dref{1.1fghyuisda} satisfies
\begin{equation}
\begin{array}{rl}
&\disp{\int_{0}^T\int_{\Omega}\left[|\nabla n_{\varepsilon}|^{\gamma_0}+ n_{\varepsilon}^{\frac{12\alpha+4}{3}}\right]\leq C(T+1),}\\
\end{array}
\label{bnmbncz2.ffghh5ghhjuyuivvbnnihjj}
\end{equation}
where $\gamma_0=\min\{3\alpha+1,2\}$.
\end{lemma}
\begin{proof}
 Due to \dref{ddfgczhhhh2.5ghju48cfg924ghyuji}, \dref{czfvgb2.5ghhjuyuccvviihjj}
and \dref{bnmbncz2.5ghhjuyuivvbnnihjj}, in light of the Gagliardo--Nirenberg inequality, for some $C_1$ and $C_2> 0$ which are independent of $\varepsilon$, one may verify that
\begin{equation}
\begin{array}{rl}
&\disp\int_{0}^T\disp\int_{\Omega} n_{\varepsilon}^{\frac{12\alpha+4}{3}} \\
=&\disp{\int_{0}^T\| { n_{\varepsilon}^{\frac{6\alpha+1}{3} }}\|^{{\frac{12\alpha+4}{6\alpha+1 }}}_{L^{\frac{12\alpha+4}{6\alpha+1 }}(\Omega)}}\\
\leq&\disp{C_{1}\int_{0}^T\left(\| \nabla{ n_{\varepsilon}^{\frac{6\alpha+1}{3}  }}\|^{2}_{L^{2}(\Omega)}\|{ n_{\varepsilon}^{\frac{6\alpha+1}{3} }}\|^{{\frac{2}{3\alpha }}}_{L^{\frac{3}{6\alpha+1}}(\Omega)}+
\|{ n_{\varepsilon}^{\frac{6\alpha+1}{3} }}\|^{{\frac{12\alpha+4}{6\alpha +1}}}_{L^{\frac{3}{6\alpha+1}}(\Omega)}\right)}\\
\leq&\disp{C_{2}(T+1)~~\mbox{for all}~~ T > 0.}\\
\end{array}
\label{ddffbnmbnddfgcz2ddfvgbhh.htt678ddfghhhyuiihjj}
\end{equation}
Case $0<\alpha\leq\frac{1}{3}:$
Therefore, employing  the
H\"{o}lder inequality (with two exponents $\frac{2}{3\alpha+1}$ and $\frac{2}{1-3\alpha}$), we conclude  that there exists a positive
 constant $C_3$ 
such that
\begin{equation}
\begin{array}{rl}
\disp\int_{0}^T\disp\int_{\Omega}|\nabla n_{\varepsilon}|^{{3\alpha+1}}
\leq&\disp{\left[\int_{0}^T\disp\int_{\Omega} n_{\varepsilon}^\frac{12\alpha-4}{3}|\nabla n_{\varepsilon}|^2\right]^{\frac{3\alpha+1}{2}}
\left[\int_{0}^T\disp\int_{\Omega} n_{\varepsilon}^{\frac{12\alpha+4}{3}}\right]^{\frac{1-3\alpha}{2}} }\\
\leq&\disp{C_{3}(T+1)~~\mbox{for all}~~ T > 0.}\\
\end{array}
\label{5555ddffbnmbncz2ddfvgffgtyybhh.htt678ghhjjjddfghhhyuiihjj}
\end{equation}

Case $\alpha\geq\frac{1}{3}:$
Multiply the first equation in $\dref{1.1fghyuisda}$ by $ n_{\varepsilon}$
 and  using $\nabla\cdot u_\varepsilon=0$, we derive
\begin{equation}
\begin{array}{rl}
&\disp{\frac{1}{{2}}\frac{d}{dt}\|{ n_{\varepsilon} }\|^{{{{2}}}}_{L^{{{2}}}(\Omega)}+
\int_{\Omega}  |\nabla n_{\varepsilon}|^2}\\
=&\disp{-
\int_{\Omega}  n_{\varepsilon}\nabla\cdot(n_{\varepsilon}\frac{1}{(1+\varepsilon n_{\varepsilon})}S_\varepsilon(x, n_{\varepsilon}, c_{\varepsilon})\cdot\nabla c_{\varepsilon})-\int_{\Omega}n_{\varepsilon}^{2}m_{\varepsilon}}\\
\leq&\disp{
\int_{\Omega}  n_{\varepsilon}\frac{1}{(1+\varepsilon n_{\varepsilon})}|S_\varepsilon(x, n_{\varepsilon}, c_{\varepsilon})||\nabla n_{\varepsilon}||\nabla c_{\varepsilon}|~~\mbox{for all}~~ t\in(0, T_{max,\varepsilon}).}
\end{array}
\label{55hhjjcffgjjjkkkkhhhjkklddfgggffgglffghhhz2.5}
\end{equation}
Now, invoke the Gagliardo-Nirenberg inequality again to obtain $C_4, C_5$ and $C_6>0$ such
that
\begin{equation}
\begin{array}{rl}
 \disp\int_{\Omega}  n_{\varepsilon}^{4-4\alpha}=&\disp{ \|n_{\varepsilon}\|^{4-4\alpha}_{L^{4-4\alpha}(\Omega)}}\\
 \leq&\disp{C_4 \|\nabla n_{\varepsilon}\|^{\frac{2(9-12\alpha)}{5}}_{L^2(\Omega)}\| n_{\varepsilon}\|^{4-4\alpha-\frac{2(9-12\alpha)}{5}}_{L^{1}(\Omega)}+C_1 \|n_{\varepsilon}\|^{4-4\alpha}_{L^1(\Omega)}}\\
  \leq&\disp{C_5( \|\nabla n_{\varepsilon}\|^{\frac{2(9-12\alpha)}{5}}_{L^2(\Omega)}+1)}\\
 \leq&\disp{C_6( \|\nabla n_{\varepsilon}\|^{2}_{L^2(\Omega)}+1)}\\
\end{array}
\label{55hherrrdddrjjssssssdddssdddcz2.5}
\end{equation}
by using \dref{ddfgczhhhh2.5ghju48cfg924ghyuji} and the  Young inequality.

Recalling \dref{x1.73142vghf48gg} and using $\alpha\geq\frac{1}{3}$, from Young inequality again, we derive that
\begin{equation}
\begin{array}{rl}
&\disp\int_{\Omega} n_{\varepsilon}\frac{1}{(1+\varepsilon n_{\varepsilon})}|S_\varepsilon(x, n_{\varepsilon}, c_{\varepsilon})||\nabla n_{\varepsilon}||\nabla c_{\varepsilon}|\\
\leq&\disp{C_S\int_{\Omega}n_{\varepsilon}^{1-\alpha} |\nabla n_{\varepsilon}||\nabla c_{\varepsilon}|}\\
\leq&\disp{\frac{1}{2}\int_{\Omega} |\nabla n_{\varepsilon}|^2+\frac{C_S^2}{2}\int_{\Omega}n_{\varepsilon}^{2(1-\alpha)}
|\nabla c_{\varepsilon}|^2}\\
\leq&\disp{\frac{1}{2}\int_{\Omega} |\nabla n_{\varepsilon}|^2+\frac{1}{4C_6}
\int_{\Omega}n_{\varepsilon}^{4(1-\alpha)}+\frac{C_6C_S^4}{4} |\nabla c_{\varepsilon}|^4~~\mbox{for all}~~ t\in(0, T_{max,\varepsilon}),}
\end{array}
\label{55hhjjcffghhhjkkllfffghggggghgghjjjhfghhhz2.5}
\end{equation}
which combined with \dref{55hhjjcffgjjjkkkkhhhjkklddfgggffgglffghhhz2.5} and \dref{55hherrrdddrjjssssssdddssdddcz2.5} implies that
\begin{equation}
\begin{array}{rl}
&\disp{\frac{1}{{2}}\frac{d}{dt}\|{ n_{\varepsilon} }\|^{{{{2}}}}_{L^{{{2}}}(\Omega)}+
\frac{3}{4}\int_{\Omega}  |\nabla n_{\varepsilon}|^2\leq\frac{C_6C_S^4}{4} |\nabla c_{\varepsilon}|^4~~\mbox{for all}~~ t\in(0, T_{max,\varepsilon}),}
\end{array}
\label{55hhjjcffghhhjkklddfgggffgglffghhhz2.5}
\end{equation}
so that, collecting \dref{bnmbncz2.5ghhjuyuivvbnnihjj} and \dref{55hhjjcffghhhjkklddfgggffgglffghhhz2.5}
 \begin{equation}
\begin{array}{rl}
&\disp{\int_{\Omega} n_{\varepsilon}^{2}\leq C_{7}~~~\mbox{for all}~~ t\in (0, T_{max,\varepsilon})}\\
\end{array}
\label{czfvgb2.5ghhjuyucfkllllfhhhhhhhjjggcvviihjj}
\end{equation}
and
\begin{equation}
\begin{array}{rl}
&\disp{\int_{0}^{ T}\int_{\Omega}  |\nabla {n_{\varepsilon}}|^2\leq C_{7}(T+1).}\\
\end{array}
\label{bnmbncz2.5ghhjugghjllllljdfghhjjyuivvbnklllnihjj}
\end{equation}
\end{proof}

In order to prove
the limit functions $n$ and $u$ gained below (see Section 6), we will rely on an
additional regularity estimate for $n_\varepsilon u_\varepsilon$.
\begin{lemma}\label{111lemmddaghjsffggggsddgghhmk4563025xxhjklojjkkk}
Let $\alpha>0$. Then there exists  
$C > 0$   independent of $\varepsilon$ such that,
for each $T\in(0, T_{max,\varepsilon})$, the solution of \dref{1.1fghyuisda} satisfies
\begin{equation}
\begin{array}{rl}
&\disp{\int_{0}^T\int_{\Omega}|n_{\varepsilon}u_\varepsilon|^{\frac{2+6\alpha}{2+3\alpha}}\leq C(T+1).}\\
\end{array}
\label{111bnmbncz2.ffghh5ghhjuyuivvbnnihjj}
\end{equation}
\end{lemma}
\begin{proof}
In  view of the H\"{o}lder inequality and the Young inequality, we have
$$
\begin{array}{rl}
&\disp\int_{0}^T\int_{\Omega}|n_{\varepsilon}u_\varepsilon|^{\frac{2+6\alpha}{2+3\alpha}} \\
\leq&\disp{\int_{0}^T\|n_{\varepsilon}\|^{\frac{2+6\alpha}{2+3\alpha}}_{L^\frac{(2+6\alpha)\theta}{(2+3\alpha)(\theta-1)}(\Omega)}
\|u_\varepsilon\|^{\frac{2+6\alpha}{2+3\alpha}}_{L^6(\Omega)}}\\
\leq&\disp{C_1\int_{0}^T\|n_{\varepsilon}\|^{\frac{2+6\alpha}{2+3\alpha}}_{L^\frac{(2+6\alpha)\theta}{(2+3\alpha)(\theta-1)}(\Omega)}
\|\nabla u_\varepsilon\|^{\frac{2+6\alpha}{2+3\alpha}}_{L^2(\Omega)}}\\
\leq&\disp{C_2\int_{0}^T\|n_{\varepsilon}\|^{2+6\alpha}_{L^\frac{(2+6\alpha)\theta}{(2+3\alpha)(\theta-1)}(\Omega)}
+C_2\int_{0}^T\|\nabla u_\varepsilon\|^{2}_{L^2(\Omega)},~\mbox{for all}~ T > 0,}\\
\end{array}
$$
where $\theta=\frac{3(2+3\alpha)}{(1+3\alpha)}.$
Next, by \dref{ddfgczhhhh2.5ghju48cfg924ghyuji}, we derive that
$$
\begin{array}{rl}
&\disp C_2\int_{0}^T\|n_{\varepsilon}\|^{2+6\alpha}_{L^\frac{(2+6\alpha)\theta}{(2+3\alpha)(\theta-1)}(\Omega)} \\
=&\disp{C_2\int_{0}^T\|n_{\varepsilon}\|^{2+6\alpha}_{L^\frac{3(2+6\alpha)}{6\alpha+5}(\Omega)}}\\
=&\disp{C_2\int_{0}^T\|n_{\varepsilon}^{\frac{6\alpha+1}{3} }\|^{\frac{3(6\alpha+2)}{6\alpha+1}}_{L^\frac{9(2+6\alpha)}{(6\alpha+5)(6\alpha+1)}(\Omega)}}\\
\leq&\disp{C_3\int_{0}^T\left(\|\nabla n_{\varepsilon}^{\frac{6\alpha+1}{3}}\|^{2}_{L^2(\Omega)}\|
 n_{\varepsilon}^{\frac{6\alpha+1}{3}}\|^{\frac{3(6\alpha+2)}{6\alpha+1}-2}_{L^{\frac{3}{6\alpha+1}}(\Omega)}
+\|
 n_{\varepsilon}^{\frac{6\alpha+1}{3}}\|^{\frac{3(6\alpha+2)}{6\alpha+1}}_{L^{\frac{3}{6\alpha+1}}(\Omega)}\right)}\\
\leq&\disp{C_4(T+1),~\mbox{for all}~ T > 0}\\
\end{array}
$$
by using \dref{ddfgczhhhh2.5ghju48cfg924ghyuji}.
\end{proof}

\section{Regularity properties of time derivatives}

To prepare our subsequent compactness properties of
$(n_\varepsilon, c_\varepsilon,m_\varepsilon, u_\varepsilon)$ by means of the Aubin-Lions lemma (see Simon \cite{Simon}), we use Lemmas \ref{fvfgfflemma45}-\ref{lemmddaghjsffggggsddgghhmk4563025xxhjklojjkkk} to obtain
the following regularity property with respect to the time variable.
\begin{lemma}\label{qqqqlemma45630hhuujjuuyytt}
Let $\alpha>0$,
\dref{dd1.1fghyuisdakkkllljjjkk} and \dref{ccvvx1.731426677gg}
 hold.
 Then for any $T>0, $
  one can find $C > 0$ independent if $\varepsilon$ such that 
\begin{equation}
 \begin{array}{ll}
\disp\int_0^T\|\partial_tn_\varepsilon(\cdot,t)\|_{({W^{1,\frac{2\alpha+6}{3\alpha}}(\Omega)})^*}^{\frac{2+6\alpha}{2+3\alpha}}dt  \leq C(T+1),\\
   \end{array}\label{1.1ddfgeddvbnmkffgghlllhyuisda}
\end{equation}
\begin{equation}
 \begin{array}{ll}
  \disp\int_0^T\|\partial_tc_\varepsilon(\cdot,t)\|_{(W^{1,4}(\Omega))^*}^{\frac{4}{3}}dt  \leq C(T+1)\\
   \end{array}\label{wwwwwqqqq1.1dddfgbhnjmdfgeddvbnmklllhyussddisda}
\end{equation}
as well as
\begin{equation}
 \begin{array}{ll}
  \disp\int_0^T\|\partial_tm_\varepsilon(\cdot,t)\|_{(W^{1,4}(\Omega))^*}^{\frac{4}{3}}dt  \leq C(T+1)\\
   \end{array}\label{wwwwwqqqq1.1dddllllfgbhnjmdfgeddvbnmklllhyussddisda}
\end{equation}
and
\begin{equation}
 \begin{array}{ll}
  \disp\int_0^T\|\partial_tu_\varepsilon(\cdot,t)\|_{(W^{1,4}_{0,\sigma}(\Omega))^*}^{\frac{4}{3}}dt  \leq C(T+1).\\
   \end{array}\label{wwwwwqqqq1.1dddfgkkllbhddffgggnjmdfgeddvbnmklllhyussddisda}
\end{equation}
\end{lemma}
\begin{proof} 
Firstly, an elementary calculation ensures that
 \begin{equation}1<\frac{2+6\alpha}{2+3\alpha}<\min\{3\alpha+1,2\}~~~\mbox{and}~~~\frac{2+6\alpha}{2+3\alpha}<\frac{12\alpha+4}{3\alpha+4}
 \label{gbhncvbmdcfvgcz2.5ghju4ddfghh8}
\end{equation}
by using $\alpha>0.$
Next,
testing the first equation of \dref{1.1fghyuisda}
 by certain   $\varphi\in C^{\infty}(\bar{\Omega})$, we have
 \begin{equation}
\begin{array}{rl}
&\disp\left|\int_{\Omega}(n_{\varepsilon,t})\varphi\right|\\
 =&\disp{\left|\int_{\Omega}\left[\Delta n_{\varepsilon}-\nabla\cdot(n_{\varepsilon}\frac{1}{(1+\varepsilon n_{\varepsilon})}S_\varepsilon(x, n_{\varepsilon}, c_{\varepsilon})\nabla c_{\varepsilon})-u_{\varepsilon}\cdot\nabla n_{\varepsilon}-n_{\varepsilon}m_{\varepsilon}\right]\varphi\right|}
\\
=&\disp{\left|\int_\Omega \left[-\nabla n_{\varepsilon}\cdot\nabla\varphi+n_{\varepsilon}\frac{1}{(1+\varepsilon n_{\varepsilon})}S_\varepsilon(x, n_{\varepsilon}, c_{\varepsilon})\nabla c_{\varepsilon}\cdot\nabla\varphi+ n_{\varepsilon}u_{\varepsilon}\cdot\nabla  \varphi- n_{\varepsilon}m_{\varepsilon}  \varphi\right]\right|}\\
\leq&\disp{\left\{\|\nabla n_{\varepsilon}\|_{L^{\frac{2+6\alpha}{2+3\alpha}}(\Omega)}+ \|\frac{ n_{\varepsilon}}{(1+\varepsilon n_{\varepsilon})}S_\varepsilon(x, n_{\varepsilon}, c_{\varepsilon})\nabla c_{\varepsilon}\|_{L^{\frac{2+6\alpha}{2+3\alpha}}(\Omega)}+ \|n_{\varepsilon}u_{\varepsilon}\|_{L^{\frac{2+6\alpha}{2+3\alpha}}(\Omega)}+\|n_{\varepsilon}m_{\varepsilon}\|_{L^{\frac{2+6\alpha}{2+3\alpha}}(\Omega)}
\right\}}\\
&\times\disp{\|\varphi\|_{W^{1,\frac{2\alpha+6}{3\alpha}}(\Omega)}}\\
\end{array}
\label{gbhncvbmdcfvgcz2.5ghju48}
\end{equation}
for all $t>0$.
Along with \dref{bnmbncz2.ffghh5ghhjuyuivvbnnihjj}, \dref{ddfgczhhhh2.5ghju48cfg924ghyuji} and \dref{1.1dddfgbhnjmdfgeddvbnmklllhyussddisda}, further implies that
\begin{equation}
\begin{array}{rl}
&\disp\int_0^T\|\partial_tn_\varepsilon(\cdot,t)\|_{({W^{1,\frac{2\alpha+6}{3\alpha}}(\Omega)})^*}^{\frac{2+6\alpha}{2+3\alpha}}dt \\
\leq&\disp{\int_0^T\left\{\|\nabla n_{\varepsilon}\|_{L^{\frac{2+6\alpha}{2+3\alpha}}(\Omega)}+ \|n_{\varepsilon}\frac{1}{(1+\varepsilon n_{\varepsilon})}S_\varepsilon(x, n_{\varepsilon}, c_{\varepsilon})\nabla c_{\varepsilon}\|_{L^{\frac{2+6\alpha}{2+3\alpha}}(\Omega)}+ \|n_{\varepsilon}u_{\varepsilon}\|_{L^{\frac{2+6\alpha}{2+3\alpha}}(\Omega)}
\right\}}^{\frac{2+6\alpha}{2+3\alpha}}dt
\\
\leq&\disp{C_1\int_0^T\left\{\|\nabla n_{\varepsilon}\|_{L^{\frac{2+6\alpha}{2+3\alpha}}(\Omega)}^{\frac{2+6\alpha}{2+3\alpha}}+ \|n_{\varepsilon}\frac{1}{(1+\varepsilon n_{\varepsilon})}S_\varepsilon(x, n_{\varepsilon}, c_{\varepsilon})\nabla c_{\varepsilon}\|_{L^{\frac{2+6\alpha}{2+3\alpha}}(\Omega)}^{\frac{2+6\alpha}{2+3\alpha}}\right\}dt}\\
&\disp{+C_1\int_0^T\left\{\|n_{\varepsilon}u_{\varepsilon}\|_{L^{\frac{2+6\alpha}{2+3\alpha}}(\Omega)}^{\frac{2+6\alpha}{2+3\alpha}}+
\|m_{\varepsilon}\|^{\frac{2+6\alpha}{2+3\alpha}}_{L^{\infty}(\Omega)}\|n_{\varepsilon}\|_{L^{\frac{2+6\alpha}{2+3\alpha}}(\Omega)}^{\frac{2+6\alpha}{2+3\alpha}}
\right\}}dt,\\
\end{array}
\label{gbhncvbmdcfvgczffghhh2.5ghju48}
\end{equation}
where $C_1$ is a positive constant independent of $\varepsilon$.
Finally,  \dref{1.1ddfgeddvbnmkffgghlllhyuisda} is a consequence of  \dref{bnmbncz2.ffghh5ghhjuyuivvbnnihjj}, \dref{1.1dddfgbhnjmdfgeddvbnmklllhyussddisda}, \dref{gbhncvbmdcfvgcz2.5ghju4ddfghh8} and the H\"{o}lder ineqaulity.
Multiplying the second equation as well as the third equation and the fourth  equation  in \dref{1.1fghyuisda} by
$\varphi\in C^{\infty}(\bar{\Omega})$, $\varphi\in C^{\infty}(\bar{\Omega})$ and  $\varphi\in C^{\infty}_{0,\sigma} (\Omega;\mathbb{R}^3)$, respectively, we obtain \dref{wwwwwqqqq1.1dddfgbhnjmdfgeddvbnmklllhyussddisda}--\dref{wwwwwqqqq1.1dddfgkkllbhddffgggnjmdfgeddvbnmklllhyussddisda}  in a completed similar manner (see \cite{Winkler51215,Zhengsddfffsdddssddddkkllssssssssdefr23} for details).

\end{proof}

\section{The proof of Theorem  \ref{theorem3}}
In order to prove Theorem  \ref{theorem3}, we first define the weak solution of four-component Keller-Segel-Navier-Stokes system \dref{1.1}.
\begin{definition}\label{df1}
Let $T > 0$ and $(n_0, c_0,m_0, u_0)$ fulfills
\dref{ccvvx1.731426677gg}.
Then a quadruple  of functions $(n, c, m, u)$ is
called a weak solution of \dref{1.1} if the following conditions are satisfied
\begin{equation}
 \left\{\begin{array}{ll}
   n\in L_{loc}^1(\bar{\Omega}\times[0,T)),\\
    c \in L_{loc}^1([0,T); W^{1,1}(\Omega)),\\
    m \in L_{loc}^1([0,T); W^{1,1}(\Omega)),\\
u \in  L_{loc}^1([0,T); W^{1,1}(\Omega);\mathbb{R}^{3}),\\
 \end{array}\right.\label{dffff1.1fghyuisdakkklll}
\end{equation}
where $n\geq 0,c\geq 0$ and $m\geq 0$ in
$\Omega\times(0, T)$ as well as $\nabla\cdot u = 0$ in the distributional sense in
 $\Omega\times(0, T)$,
moreover,
\begin{equation}\label{726291hh}
\begin{array}{rl}
 &u\otimes u \in L^1_{loc}(\bar{\Omega}\times [0, \infty);\mathbb{R}^{3\times 3})~~\mbox{and}~~~ nm~\mbox{belong to}~~ L^1_{loc}(\bar{\Omega}\times [0, \infty)),\\
  &cu,~ ~nu,~~mu ~~\mbox{and}~~nS(x,n,c)\nabla c~ \mbox{belong to}~~
L^1_{loc}(\bar{\Omega}\times [0, \infty);\mathbb{R}^{3})
\end{array}
\end{equation}
and
\begin{equation}
\begin{array}{rl}\label{eqx45xx12112ccgghh}
\disp{-\int_0^{T}\int_{\Omega}n\varphi_t-\int_{\Omega}n_0\varphi(\cdot,0)  }=&\disp{-
\int_0^T\int_{\Omega}\nabla n\cdot\nabla\varphi+\int_0^T\int_{\Omega}n
S(x,n,c)\nabla c\cdot\nabla\varphi}\\
&+\disp{\int_0^T\int_{\Omega}nu\cdot\nabla\varphi-\int_0^T\int_{\Omega}nm\varphi}\\
\end{array}
\end{equation}
for any $\varphi\in C_0^{\infty} (\bar{\Omega}\times[0, T))$ satisfying
 $\frac{\partial\varphi}{\partial\nu}= 0$ on $\partial\Omega\times (0, T)$
  as well as
  \begin{equation}
\begin{array}{rl}\label{eqx45xx12112ccgghhjj}
\disp{-\int_0^{T}\int_{\Omega}c\varphi_t-\int_{\Omega}c_0\varphi(\cdot,0)  }=&\disp{-
\int_0^T\int_{\Omega}\nabla c\cdot\nabla\varphi-\int_0^T\int_{\Omega}c\varphi+\int_0^T\int_{\Omega}m\varphi+
\int_0^T\int_{\Omega}cu\cdot\nabla\varphi}\\
\end{array}
\end{equation}
for any $\varphi\in C_0^{\infty} (\bar{\Omega}\times[0, T))$,
\begin{equation}
\begin{array}{rl}\label{eqx45fffffxx12112ccgghhjj}
\disp{-\int_0^{T}\int_{\Omega}m\varphi_t-\int_{\Omega}m_0\varphi(\cdot,0)  }=&\disp{-
\int_0^T\int_{\Omega}\nabla m\cdot\nabla\varphi-\int_0^T\int_{\Omega}nm\varphi+
\int_0^T\int_{\Omega}mu\cdot\nabla\varphi}\\
\end{array}
\end{equation}
for any $\varphi\in C_0^{\infty} (\bar{\Omega}\times[0, T))$
 and
\begin{equation}
\begin{array}{rl}\label{eqx45xx12112ccgghhjjgghh}
\disp{-\int_0^{T}\int_{\Omega}u\varphi_t-\int_{\Omega}u_0\varphi(\cdot,0) -\kappa
\int_0^T\int_{\Omega} u\otimes u\cdot\nabla\varphi }=&\disp{-
\int_0^T\int_{\Omega}\nabla u\cdot\nabla\varphi-
\int_0^T\int_{\Omega}(n+m)\nabla\phi\cdot\varphi}\\
\end{array}
\end{equation}
for any $\varphi\in C_0^{\infty} (\bar{\Omega}\times[0, T);\mathbb{R}^3)$ fulfilling
$\nabla\varphi\equiv 0$ in
 $\Omega\times(0, T)$.
 If $\Omega\times (0,\infty)\longrightarrow \mathbb{R}^6$ is a weak solution of \dref{1.1} in
 $\Omega\times(0, T)$ for all $T > 0$, then we call
$(n, c, m, u)$ a global weak solution of \dref{1.1}.
\end{definition}
With the help of a priori estimates (see Lemmas \ref{lemmaghjffggssddgghhmk4563025xxhjklojjkkk}--\ref{111lemmddaghjsffggggsddgghhmk4563025xxhjklojjkkk} and \ref{qqqqlemma45630hhuujjuuyytt}), 
by extracting suitable subsequences in a standard
way (see also \cite{Winkler51215}), we could see the solution of \dref{1.1} is indeed globally solvable.


{\bf The proof of Theorem  \ref{theorem3}}
\begin{proof}
Firstly, due to \dref{czfvgb2.5ghhjuyuccvviihjj}
and \dref{bnmbncz2.5ghhjuyuivvbnnihjj}, in light of the Gagliardo--Nirenberg inequality, we derive that there exist positive constants $C_{1}$ and $C_{2}$  such that
\begin{equation}
\begin{array}{rl}
\disp\int_{0}^T\disp\int_{\Omega} |u_{\varepsilon}|^{\frac{10}{3}} \leq&\disp{C_{1}\int_{0}^T\left(\| \nabla{ u_{\varepsilon}}\|^{2}_{L^{2}(\Omega)}\|{ u_{\varepsilon}}\|^{{\frac{4}{3}}}_{L^{2}(\Omega)}+
\|{ u_{\varepsilon}}\|^{{\frac{10}{3}}}_{L^{2}(\Omega)}\right)}\\
\leq&\disp{C_{2}(T+1)~~\mbox{for all}~~ T > 0,}\\
\end{array}
\label{ddffbnmbnddfgffggjjkkuuiicz2dvgbhh.t8ddhhhyuiihjj}
\end{equation}
so that, according to Lemmas \ref{fvfgfflemma45}, \ref{lemmaghjffggssddgghhmk4563025xxhjklojjkkk} and \ref{qqqqlemma45630hhuujjuuyytt},  an application of the Aubin--Lions lemma (see e.g. \cite{Simon})
provides a sequence $(\varepsilon_j)_{j\in \mathbb{N}}\subset (0, 1)$ and limit functions $n,c,m$ and $u$
 such that $\varepsilon_j\searrow 0$ as $j\rightarrow\infty$ and such that
hold as well as
\begin{equation}
n_\varepsilon\rightarrow n ~~\mbox{a.e.}~~\mbox{in}~~\Omega\times(0,\infty)~~\mbox{and in}~~ L_{loc}^{r}(\bar{\Omega}\times[0,\infty))~~\mbox{with}~~r= \left\{\begin{array}{ll}
 {3\alpha+1}~~\mbox{if}~~0<\alpha<\frac{1}{3},\\
  2~~\mbox{if}~~\alpha\geq\frac{1}{3},\\
   \end{array}\right.\label{zjscz2.5297x963ddfgh0ddfggg6662222tt3}
\end{equation}
\begin{equation}
\nabla n_\varepsilon\rightharpoonup \nabla n ~~\mbox{in}~~ L_{loc}^{r}(\bar{\Omega}\times[0,\infty))~~\mbox{with}~~r= \left\{\begin{array}{ll}
 {3\alpha+1}~~\mbox{if}~~0<\alpha<\frac{1}{3},\\
  2~~\mbox{if}~~\alpha\geq\frac{1}{3},\\
   \end{array}\right.\label{zjscz2.5297x963ddfgh0ddgghjjfggg6662222tt3}
\end{equation}
\begin{equation}
c_\varepsilon\rightarrow c ~~\mbox{in}~~ L^{2}_{loc}(\bar{\Omega}\times[0,\infty))~~\mbox{and}~~\mbox{a.e.}~~\mbox{in}~~\Omega\times(0,\infty),
 \label{zjscz2.fgghh5297x963ddfgh0ddfggg6662222tt3}
\end{equation}
\begin{equation}
m_\varepsilon\rightarrow m ~~\mbox{in}~~ L^{2}_{loc}(\bar{\Omega}\times[0,\infty))~~\mbox{and}~~\mbox{a.e.}~~\mbox{in}~~\Omega\times(0,\infty),
 \label{zjscz2.fgghh5297x96ddddd3ddfgh0ddfggg6662222tt3}
\end{equation}
\begin{equation}
\nabla c_\varepsilon\rightharpoonup \nabla c ~~\mbox{in}~~ L^{4}_{loc}(\bar{\Omega}\times[0,\infty)),
 \label{zjscz2.fgghh5297x963ddfgh0dddddfggg6662222tt3}
\end{equation}

\begin{equation}
u_\varepsilon\rightarrow u~~\mbox{in}~~ L_{loc}^2(\bar{\Omega}\times[0,\infty))~~\mbox{and}~~\mbox{a.e.}~~\mbox{in}~~\Omega\times(0,\infty),
 \label{zjscz2.5297x96302222t666t4}
\end{equation}
\begin{equation}
\nabla c_\varepsilon\rightharpoonup \nabla c~~\begin{array}{ll}
 \mbox{in}~~ L_{loc}^{2}(\bar{\Omega}\times[0,\infty)),
   \end{array}\label{1.1ddfgghhhge666ccdf2345ddvbnmklllhyuisda}
\end{equation}
\begin{equation}
\nabla m_\varepsilon\rightharpoonup \nabla m~~\begin{array}{ll}
 \mbox{in}~~ L_{loc}^{2}(\bar{\Omega}\times[0,\infty))
   \end{array}\label{1.1dddddfgghhhge666ccdf2345ddvbnmklllhyuisda}
\end{equation}
as well as
\begin{equation}
 \nabla u_\varepsilon\rightharpoonup \nabla u ~~\mbox{ in}~~L^{2}_{loc}(\bar{\Omega}\times[0,\infty);\mathbb{R}^{3})
 \label{zjscz2.5297x96366602222tt4455}
\end{equation}
and
\begin{equation}
 u_\varepsilon\rightharpoonup u ~~\mbox{ in}~~L^{\frac{10}{3}}_{loc}(\bar{\Omega}\times[0,\infty))
 \label{zjscz2.5ffgtt297x96302266622tt4}
\end{equation}
 with some quadruple  $(n, c,m, u)$.

 In the following, we shall prove $(n,c,m,u)$ is a weak solution of problem \dref{1.1} in Definition \ref{df1}.
To this end,  recalling \dref{ddfgczhhhh2.5ghju48cfg924ghyuji}, \dref{czfvgb2.5ghhjuyuccvviihjj} and \dref{bnmbncz2.5ghhjuyuivvbnnihjj},  we derive 
$(c_{\varepsilon})_{\varepsilon\in(0,1)}$ is bounded in
$L^{2} ((0, T); W^{2,2}(\Omega))$.
Thus,  by virtue of \dref{wwwwwqqqq1.1dddfgbhnjmdfgeddvbnmklllhyussddisda} and the Aubin--Lions lemma we derive that  the relative compactness of $(c_{\varepsilon})_{\varepsilon\in(0,1)}$ in
$L^{2} ((0, T); W^{1,2}(\Omega))$. We can pick an appropriate subsequence which is
still written as $(\varepsilon_j )_{j\in \mathbb{N}}$ such that $\nabla c_{\varepsilon_j} \rightarrow z_1$
 in $L^{2} (\Omega\times(0, T))$ for all $T\in(0, \infty)$ and some
$z_1\in L^{2} (\Omega\times(0, T))$ as $j\rightarrow\infty$, hence $\nabla c_{\varepsilon_j} \rightarrow z_1$ a.e. in $\Omega\times(0, \infty)$
 as $j \rightarrow\infty$.
In view  of \dref{1.1ddfgghhhge666ccdf2345ddvbnmklllhyuisda} and  the Egorov theorem we conclude  that
$z_1=\nabla c,$ and whence
\begin{equation}
\nabla c_\varepsilon\rightarrow \nabla c ~~\mbox{a.e.}~~\mbox{in}~~\Omega\times(0,\infty)
 \label{1.1ddhhyujiiifgghhhge666ccdf2345ddvbnmklllhyuisda}
\end{equation}
 holds.

Next, $\alpha>0$ yields to
$$
r>1,
   $$
   where $r$ is given by \dref{zjscz2.5297x963ddfgh0ddfggg6662222tt3}.
   Therefore,
with the help of  \dref{zjscz2.5297x963ddfgh0ddfggg6662222tt3}--\dref{zjscz2.fgghh5297x963ddfgh0ddfggg6662222tt3}, \dref{zjscz2.5297x96302222t666t4}--\dref{zjscz2.5297x96366602222tt4455}, we can derive  \dref{dffff1.1fghyuisdakkklll}.
Now, by the nonnegativity of $n_\varepsilon$,  $c_\varepsilon$ and $m_\varepsilon$, we derive  $n,c \geq 0$ and $m\geq 0$. Next, due to
\dref{zjscz2.5297x96366602222tt4455} and $\nabla\cdot u_{\varepsilon} = 0$, we conclude that
$\nabla\cdot u = 0$ a.e. in $\Omega\times (0, \infty)$.
Now, by \dref{x1.73142vghf48gg}, \dref{bnmbncz2.ffghh5ghhjuyuivvbnnihjj}
and \dref{3.10gghhjuuloollyuigghhhyy}, we derive that
$$n_\varepsilon \frac{1}{(1+\varepsilon n_{\varepsilon})}S_\varepsilon(x, n_{\varepsilon}, c_{\varepsilon})\leq C_Sn_{\varepsilon}.$$
It is not difficult to verify that
$$\frac{3\alpha+4}{12\alpha+4}=\frac{1}{4}+\frac{3}{12\alpha+4}.$$
From this and
by \dref{bnmbncz2.ffghh5ghhjuyuivvbnnihjj} and \dref{bnmbncz2.5ghhjuyuivvbnnihjj}, and recalling the H\"{o}lder inequality, we can obtain for some positive constant $C_1$ such that
\begin{equation}
 \begin{array}{ll}
  \disp\int_0^T\int_{\Omega}\left[|n_\varepsilon \frac{1}{(1+\varepsilon n_{\varepsilon})}S_\varepsilon(x, n_{\varepsilon}, c_{\varepsilon})\nabla c_\varepsilon|^{\frac{12\alpha+4}{3\alpha+4}} \right]\leq C_1(T+1),\\
   \end{array}\label{1.1dddfgbhnjmdfgeddvbnmklllhyussddisda}
\end{equation}
so that,
 we conclude that
\begin{equation}n_\varepsilon \frac{1}{(1+\varepsilon n_{\varepsilon})}S_\varepsilon(x, n_{\varepsilon}, c_{\varepsilon})\nabla c_\varepsilon\rightharpoonup z_2
~~\mbox{in}~~ L^{\frac{12\alpha+4}{3\alpha+4}}(\Omega\times(0,T);\mathbb{R}^{3})~~\mbox{as}~~\varepsilon = \varepsilon_j\searrow 0~~\mbox{for each}~~ T\in(0,\infty).
\label{1.1ddddfddffttyygghhyujiiifgghhhgffgge6bhhjh66ccdf2345ddvbnmklllhyuisda}
\end{equation}
Next, it follows from \dref{x1.73142vghf48rtgyhu}, \dref{3.10gghhjuuloollyuigghhhyy}, \dref{zjscz2.5297x963ddfgh0ddfggg6662222tt3}, \dref{zjscz2.fgghh5297x963ddfgh0ddfggg6662222tt3} and \dref{1.1ddhhyujiiifgghhhge666ccdf2345ddvbnmklllhyuisda} that
\begin{equation}n_\varepsilon \frac{1}{(1+\varepsilon n_{\varepsilon})}S_\varepsilon(x, n_{\varepsilon}, c_{\varepsilon})\nabla c_\varepsilon\rightarrow nS(x, n, c)\nabla c~~\mbox{a.e.}~~\mbox{in}~~\Omega\times(0,\infty)~~\mbox{as}~~\varepsilon = \varepsilon_j\searrow 0.
\label{1.1ddddfddfftffghhhtyygghhyujiiifgghhhgffgge6bhhjh66ccdf2345ddvbnmklllhyuisda}
\end{equation}
Again by the Egorov theorem, we gain $z_2=nS(x, n, c)\nabla c,$ and hence \dref{1.1ddddfddffttyygghhyujiiifgghhhgffgge6bhhjh66ccdf2345ddvbnmklllhyuisda}
can be rewritten as
\begin{equation}n_\varepsilon \frac{1}{(1+\varepsilon n_{\varepsilon})}S_\varepsilon(x, n_{\varepsilon}, c_{\varepsilon})\nabla c_\varepsilon\rightharpoonup nS(x, n, c)\nabla c
~~\mbox{in}~~ L^{\frac{12\alpha+4}{3\alpha+4}}(\Omega\times(0,T);\mathbb{R}^{3})~~\mbox{as}~~\varepsilon = \varepsilon_j\searrow 0
\label{1.1ddddfddffttyygghhyujiiffghhhifgghhhgffgge6bhhjh66ccdf2345ddvbnmklllhyuisda}
\end{equation}
for each $T\in(0,\infty)$.
This together with ${\frac{12\alpha+4}{3\alpha+4}}>1$ (by $\alpha>0$)
implies  the integrability of $nS(x, n, c)\nabla c$ in \dref{726291hh} as well.
It is not hard to check that
$${\frac{2+6\alpha}{2+3\alpha}}>1~~\mbox{by}~~~\alpha>0.
   $$
   Thereupon, recalling \dref{111bnmbncz2.ffghh5ghhjuyuivvbnnihjj}, we infer that for each $T\in(0, \infty)$
  \begin{equation}
n_\varepsilon u_{\varepsilon}\rightharpoonup z_3 ~~\mbox{in}~~ L^{\frac{2+6\alpha}{2+3\alpha}}(\Omega\times(0,T))
   ~~\mbox{as}~~\varepsilon = \varepsilon_j\searrow 0.
   \label{zjscz2.529ffghhh7x963ddfgh0ddfggg6662222tt3}
\end{equation}
This, together with \dref{zjscz2.5297x963ddfgh0ddfggg6662222tt3} and \dref{zjscz2.5297x96302222t666t4}, implies
 \begin{equation}n_\varepsilon u_\varepsilon\rightarrow nu~~\mbox{a.e.}~~\mbox{in}~~\Omega\times(0,\infty)~~\mbox{as}~~\varepsilon = \varepsilon_j\searrow 0.
\label{1.1ddddfddfftffghhhtyygghhyujiiifgghhhgffgge6bhhffgggjh66ccdf2345ddvbnmklllhyuisda}
\end{equation}
 Along with \dref{zjscz2.529ffghhh7x963ddfgh0ddfggg6662222tt3} and \dref{1.1ddddfddfftffghhhtyygghhyujiiifgghhhgffgge6bhhffgggjh66ccdf2345ddvbnmklllhyuisda}, the Egorov theorem guarantees that $z_3=nu$, whereupon we derive from \dref{zjscz2.529ffghhh7x963ddfgh0ddfggg6662222tt3} that
 \begin{equation}
n_\varepsilon u_{\varepsilon}\rightharpoonup nu ~~\mbox{in}~~ L^{\frac{2+6\alpha}{2+3\alpha}}(\Omega\times(0,T))
   ~~\mbox{as}~~\varepsilon = \varepsilon_j\searrow 0\label{zjscz2.529ffghhh7x963djkkkkdfgh0ddfggg6662222tt3ffff}
\end{equation}
   for each $T\in(0, \infty)$.

 By a similar argument as in the proof of \dref{zjscz2.529ffghhh7x963djkkkkdfgh0ddfggg6662222tt3ffff}, one can derive from \dref{ddfgczhhhh2.5ghju48cfg924ghyuji}, \dref{bnmbncz2.ffghh5ghhjuyuivvbnnihjj} as well as \dref{ddfgczhhhh2.5ghju48cfg924ghyuji} and \dref{zjscz2.5297x963ddfgh0ddfggg6662222tt3} and \dref{zjscz2.fgghh5297x96ddddd3ddfgh0ddfggg6662222tt3} that
  \begin{equation}
n_\varepsilon m_{\varepsilon}\rightharpoonup nm~~\mbox{in}~~ L^{\frac{4}{3}}(\Omega\times(0,T))
   ~~\mbox{as}~~\varepsilon = \varepsilon_j\searrow 0\label{jjjzjscz2.529ffghhh7x963djkkkkdfgh0ddfggg6662222tt3}
\end{equation}
   for each $T\in(0, \infty)$.

As a straightforward consequence of \dref{zjscz2.fgghh5297x963ddfgh0ddfggg6662222tt3}, \dref{zjscz2.fgghh5297x96ddddd3ddfgh0ddfggg6662222tt3} and \dref{zjscz2.5297x96302222t666t4}, it holds that
\begin{equation}
c_\varepsilon u_\varepsilon\rightarrow cu ~~\mbox{ in}~~ L^{1}_{loc}(\bar{\Omega}\times(0,\infty);\mathbb{R}^{3})~~~\mbox{as}~~\varepsilon=\varepsilon_j\searrow0
 \label{zxxcvvfgggjscddfffcvvfggz2.5fff297x96302222tt4}
\end{equation}
and
\begin{equation}
m_\varepsilon u_\varepsilon\rightarrow mu ~~\mbox{ in}~~ L^{1}_{loc}(\bar{\Omega}\times(0,\infty);\mathbb{R}^{3})~~~\mbox{as}~~\varepsilon=\varepsilon_j\searrow0.
 \label{zxxcvvfgggjscddfffcvvfggz2.5297x96302222tt4}
\end{equation}
Thus, the integrability of $nu,nm,mu$ and $cu$ in \dref{726291hh} is verified by \dref{zjscz2.529ffghhh7x963djkkkkdfgh0ddfggg6662222tt3ffff}--\dref{zxxcvvfgggjscddfffcvvfggz2.5297x96302222tt4}.
Now, following an argument from Lemma 4.1 of \cite{Winkler51215} (see also \cite{Zhengsddfffsdddssddddkkllssssssssdefr23}), one could prove
\begin{equation}
\begin{array}{rl}
Y_{\varepsilon}u_{\varepsilon}\otimes u_{\varepsilon}\rightarrow u \otimes u ~~\mbox{in}~~L^1_{loc}(\bar{\Omega}\times[0,\infty);\mathbb{R}^{3\times 3})~~\mbox{as}~~\varepsilon=\varepsilon_j\searrow0.
\end{array}
\label{ggjjssdffzccfccvvfgghjjjvvvvgccvvvgjscz2.5297x963ccvbb111kkuu}
\end{equation}
 Finally, according to \dref{zjscz2.5297x963ddfgh0ddfggg6662222tt3}--\dref{zjscz2.fgghh5297x96ddddd3ddfgh0ddfggg6662222tt3},
 \dref{zjscz2.5297x96302222t666t4}--\dref{zjscz2.5297x96366602222tt4455}, \dref{1.1ddfgghhhge666ccdf2345ddvbnmklllhyuisda},
 \dref{jjjzjscz2.529ffghhh7x963djkkkkdfgh0ddfggg6662222tt3}--\dref{ggjjssdffzccfccvvfgghjjjvvvvgccvvvgjscz2.5297x963ccvbb111kkuu}, we may pass to the limit in
the respective weak formulations associated with the the regularized system \dref{1.1fghyuisda} and get
 the integral
identities \dref{eqx45xx12112ccgghh}--\dref{eqx45xx12112ccgghhjjgghh}.
\end{proof}

{\bf Acknowledgement}:
This work is partially supported by  the National Natural
Science Foundation of China (No. 11601215), Shandong Provincial
Science Foundation for Outstanding Youth (No. ZR2018JL005) and Project funded by China
Postdoctoral Science Foundation (No. 2019M650927).

\end{document}